\documentclass[reqno]{amsart}
\usepackage{amssymb, amsmath, amsthm, amscd, mathrsfs}

\usepackage{paralist}
 
\usepackage{bigints}
\usepackage{empheq}
\usepackage{amssymb}
\usepackage{cases}
\usepackage{enumitem}
\allowdisplaybreaks

\usepackage[colorlinks]{hyperref}
\usepackage[colorinlistoftodos]{todonotes}

\theoremstyle{plain}
\newtheorem{theorem}{Theorem}[section]
\newtheorem{corollary}{Corollary}

\newtheorem{lemma}[theorem]{Lemma}
\newtheorem{proposition}{Proposition}

\theoremstyle{definition}
\newtheorem{definition}[theorem]{Definition}
\newtheorem{remark}{Remark}

\def \dv {\mathrm{div}}
\def \d {\mathrm{d}}

\title[Lipschitz stability for an inverse source problem] 
      {Lipschitz stability for an inverse source problem in anisotropic parabolic equations with dynamic boundary conditions}

\author{E. M. Ait Ben Hassi}
\address{E. M. Ait Ben hassi, Cadi Ayyad University, Faculty of Sciences Semlalia, 2390, Marrakesh, Morocco}
\email{m.benhassi@gmail.com}	

\author{S. E. Chorfi}
\address{S. E. Chorfi, Cadi Ayyad University, Faculty of Sciences Semlalia, 2390, Marrakesh, Morocco}
\email{chorphi@gmail.com}	

\author{L. Maniar}
\address{L. Maniar, Cadi Ayyad University, Faculty of Sciences Semlalia, 2390, Marrakesh, Morocco}
\email{maniar@uca.ma}

\author{O. Oukdach}
\address{O. Oukdach, Cadi Ayyad University, Faculty of Sciences Semlalia, 2390, Marrakesh, Morocco}
\email{omar.oukdach@gmail.com}


\begin{document}
\begin{abstract}
In this paper, we study an inverse problem for linear parabolic system with variable diffusion coefficients subject to dynamic boundary conditions. We prove a global Lipschitz stability for the inverse problem involving a simultaneous recovery of two source terms from a single measurement and interior observations, based on a recent Carleman estimate for such problems.
\end{abstract}

\maketitle

\section{Introduction and statement of the problem}
We are interested in the inverse source problem for linear parabolic system with variable diffusion coefficients and dynamic boundary conditions in bounded domains. It consists of recovering two source terms from a single measurement of the temperature at a given time with an additional internal observation on the solution localized in a small region of the physical domain.\\
To introduce the problem, let $T>0$ and $\Omega \subset \mathbb{R}^N$ a bounded domain, $N\geq 2$, with smooth boundary $\Gamma=\partial\Omega$ of class $C^2$, and outer unit normal field $\nu$ on $\Gamma$ be given. We denote $\Omega_T =(0,T)\times \Omega, \quad \omega_T =(0,T)\times \omega, \quad \Gamma_T =(0,T)\times \Gamma$, where $\omega \Subset \Omega$ is a nonempty open subset. Consider the following system
\begin{empheq}[left = \empheqlbrace]{alignat=2}
\begin{aligned}
&\partial_t y - \dv(A(x) \nabla y) + B\cdot \nabla y + p(x)y = F(t,x) &\text{in } \Omega_T , \\
&\partial_t y_{\Gamma} - \dv_{\Gamma} (D(x)\nabla_\Gamma y_{\Gamma}) +\partial_{\nu}^A y + \langle b, \nabla_\Gamma y_{\Gamma} \rangle_{\Gamma} + q(x)y_{\Gamma} = G(t,x) &\text{on } \Gamma_T, \\
&y_{\Gamma}(t,x) = y\rvert_{\Gamma}(t,x) &\text{on } \Gamma_T, \\
&(y,y_{\Gamma})\rvert_{t=0}=(y_0, y_{0,\Gamma})   &\Omega\times\Gamma. \label{eq1to4}
\end{aligned}
\end{empheq}
The initial states are denoted by $(y_0, y_{0,\Gamma})\in L^2(\Omega)\times L^2(\Gamma)$, while the source terms are $F\in L^2(\Omega_T)$ and $G\in L^2(\Gamma_T)$. All the coefficients in system \eqref{eq1to4} are assumed to be bounded,
\begin{equation}
B \in L^\infty(\Omega)^N, \quad p \in L^\infty(\Omega), \quad b\in L^\infty(\Gamma)^N, \; q \in L^\infty(\Gamma). \label{bdcoeff}
\end{equation}
We assume that the diffusion matrices $A$ and $D$ are symmetric and uniformly elliptic, i.e.,
\begin{align}
A=(a_{ij})_{i,j} \in C^1(\overline{\Omega}; \mathbb{R}^{N\times N}), \quad a_{ij}=a_{ji}, \quad 1\leq i,j\leq N, \label{symA}\\
D=(d_{ij})_{i,j} \in C^1(\Gamma; \mathbb{R}^{N\times N}), \quad d_{ij}=d_{ji}, \quad 1\leq i,j\leq N, \label{symD}
\end{align}
and there exists a constant $\beta_0 >0$ such that
\begin{align}
&\langle A(x)\zeta, \zeta\rangle \geq \beta_0 |\zeta|^2, \qquad\qquad x\in \overline{\Omega}, \;\zeta \in \mathbb{R}^N, \label{uellipA}\\
&\langle D(x)\zeta, \zeta \rangle_{\Gamma} \geq \beta_0 |\zeta|_\Gamma^2, \;\quad\qquad x\in \Gamma, \;\zeta \in \mathbb{R}^N, \label{uellipD}
\end{align}
where $\langle \cdot, \cdot\rangle$ is the Euclidean inner product (also denoted by $``\cdot"$) and $\langle \cdot, \cdot\rangle_\Gamma$ is the Riemannian inner product on $\Gamma$ as defined below.\\
We denote $y\rvert_\Gamma$ the trace of $y$. The conormal derivative with respect to $A$ is given by
$$\partial_{\nu}^A y :=(A\nabla y\cdot \nu) \rvert_\Gamma=\sum\limits_{i,j=1}^N a_{ij}(x) (\partial_i y)\rvert_\Gamma\nu_j.$$
For the identity matrix, the normal derivative is $\partial_\nu y:=(\nabla y\cdot \nu) \rvert_\Gamma$.\\
Here, $\dv$ denotes the divergence operator with respect to the space variable in $\Omega$.\\
The boundary $\Gamma$ is considered to be a $(N-1)$-dimensional compact Riemannian submanifold, without boundary. Let $g$ be the Riemannian metric on $\Gamma$ induced by the natural embedding $\Gamma \hookrightarrow \mathbb{R}^N$. We fix a coordinate system $x=(x^j)$ and we denote by $\displaystyle \left(\frac{\partial}{\partial x^j}\right)$ the corresponding tangent vector field. In local coordinates, $g$ is given by $g_{ij}:=\left\langle \dfrac{\partial }{\partial x^i}, \dfrac{\partial }{\partial x^j}\right\rangle$. We define the tangential gradient locally for any smooth function $y$ on $\Gamma$ by
$$\nabla_\Gamma y :=\sum_{i,j=1}^{N-1} g^{ij} \frac{\partial y}{\partial x^j} \frac{\partial }{\partial x^i},$$
where we denote $g=(g_{ij})$, $(g^{ij})$ its inverse and $|g|=\det(g_{ij})$. It is well-known that $\nabla_\Gamma y$ is the projection of the standard Euclidean gradient $\nabla y$ onto the tangent space on $\Gamma$, that is,
\begin{equation}
\nabla_\Gamma y = \nabla y -\langle \nabla y,\nu\rangle \nu. \label{eqtgrad}
\end{equation}
The divergence operator $\dv_\Gamma$ associated with the Riemannian metric $g$ is defined locally as follows
\begin{equation*}
\dv_\Gamma(X)=\frac{1}{\sqrt{|g|}} \sum_{j=1}^{N-1} \frac{\partial}{\partial x^j} \left(\sqrt{|g|}\, X^j\right), \quad X = \sum_{j=1}^{N-1} \frac{\partial}{\partial x^j} X^j. 
\end{equation*}
For any $x\in \Gamma$, the inner product and the norm on the tangent space $T_x \Gamma$ are given by
$$g(X_1, X_2)=\langle X_1, X_2 \rangle_\Gamma =\sum_{i,j=1}^{N-1} g_{ij} X_1^i X_2^j, \qquad |X|_\Gamma=\langle X,X\rangle_\Gamma^{1/2}.$$
Then, the Laplace-Beltrami operator $\Delta_\Gamma$ associated to $g$ is given by
\begin{equation*}
\Delta_\Gamma =\dv_\Gamma(\nabla_\Gamma)=\frac{1}{\sqrt{|g|}} \sum_{i,j=1}^{N-1} \frac{\partial}{\partial x^i} \left(\sqrt{|g|}\, g^{ij} \frac{\partial}{\partial x^j}\right).
\end{equation*}
Since $\Gamma$ is a compact Riemannian manifold without boundary, the following divergence formula holds
\begin{equation}
\int_\Gamma (\dv_\Gamma X)z \,\d S =- \int_\Gamma \langle X, \nabla_\Gamma z \rangle_\Gamma \,\d S, \qquad z\in H^1(\Gamma), \label{sdt}
\end{equation}
where $X$ is any $C^1$ vector field on $\Gamma$ and $\d S$ denotes the surface measure on $\Gamma$.\\
We refer to \cite{Go'06} for the physical interpretation and derivations of the dynamic boundary condition $\eqref{eq1to4}_2$.

A number of authors have studied evolution equations with dynamic boundary conditions from different mathematical aspects, see for instance \cite{CT'17,FG'02,Go'06,MMS'17,MZ'05,VV'11}. The main ingredient to establish various facts in control theory as well as inverse problems is Carleman estimates, which are, roughly speaking, some $L^2$-weighted inequalities estimating the solutions of PDEs in terms of associated differential operators, using large parameters and appropriate weight functions.\\
Recently, Maniar et al. have established a new Carleman estimate \cite{MMS'17} in context of null controllability of system \eqref{eq1to4} with constant diffusion matrices and without drift terms, i.e., $A=dI, D=\delta I$ and $B=b=0$, for some constants $d,\delta >0$. The drift terms case has been recently studied in \cite{KM'19, KMMR'19}.\\

\noindent\textbf{Inverse Source Problem.}\\
Let $T>0$, $t_0 \in (0,T)$, $\displaystyle T_0=\frac{T+t_0}{2}$ and $\mathbb{L}^2:=L^2(\Omega)\times L^2(\Gamma)$.\\
For a given $C_0>0$, we introduce the set of admissible source terms as follows
\begin{align}
\begin{aligned}
\mathcal{S}(C_0) := \{
        (F,G)\in H^1(0,T; \mathbb{L}^2) &\colon |F_t(t,x)| \leq C_0 |F(T_0,x)|, \text{ a.e. }(t,x)\in \Omega_T, \\
       \hspace{3.1cm} \text{ and } &|G_t(t,x)| \leq C_0 |G(T_0,x)|, \text{ a.e. }(t,x)\in \Gamma_T
\}. \label{eqas}
\end{aligned}
\end{align}
Our purpose is to determine the couple of source terms $\mathcal{F}=(F,G)$ in \eqref{eq1to4} belonging to $\mathcal{S}(C_0)$, from a single measurement $Y(T_0, \cdot)=(y,y_{\Gamma})\rvert_{t=T_0}$ and extra partial observation on the first component of the solution, namely, $y\rvert_{(t_0,T)\times\omega}$.\\
We mainly aim to establish a global Lipschitz stability for the source terms in \eqref{eq1to4}. In the above inverse source problem, if we only consider the single measurement of the temperature $Y(T_0, \cdot)=(y,y_{\Gamma})\rvert_{t=T_0}$ as observation, the inverse problem becomes ill-posed in the sense of Hadamard, due to compactness reasons. Hence, the additional observation $y\rvert_{(t_0,T)\times\omega}$ is important to overcome the instability. The internal observation regions required by the Carleman estimate approach are often of the form $\mathcal{O} :=(\{T_0\}\times \Omega) \cup ((t_0,T)\times\omega).$

A new approach with quite realistic observations was introduced in \cite{CR'10} for a uniqueness result in an inverse parabolic problem, where the observation is taken on a single point $x_0 \in \Omega$ of the spacial domain, and any small time interval $(0,t_0)$, that is, $\mathcal{O}_0 := (0,t_0)\times \{x_0\}$. However, up to our knowledge, no stability result was proven by this method. This type of inverse source problems was studied by many researchers in the case of static boundary conditions of Dirichlet, Neumann or Robin types, see for instance \cite{BK'81, IY'98, IS'98, IS'90, KL'92}. For a general review on inverse parabolic problems by Carleman estimate we refer to \cite{YA'09}, and the recent book \cite{BY'17} for inverse hyperbolic problems.

Our result is based on a new  Carleman estimate for the system \eqref{eq1to4}, which extends the one proved in \cite{MMS'17} for the system with standard Laplace and Laplace-Beltrami operators, to a general second order elliptic operators in divergence form; applied with the pioneering idea of applying such estimate to inverse problems, originally proposed by Bukhgeim and Klibanov in \cite{BK'81}. This new approach allows to prove uniqueness and H\"older stability result using a local Carleman estimate. In 1998, Imanuvilov and Yamamoto in \cite{IY'98} adapted the idea of Bukhgeim and Klibanov with global Carleman estimate proved by Fursikov and Imanuvilov in \cite{FI'96}. This permits to improve the H\"older stability to a global Lipschitz stability for inverse source problems with classical boundary conditions in the parabolic case. The aforementioned method has also been successfully applied to degenerate/singular parabolic equations and coupled systems, see e.g., \cite{BFM'16,BFM'14,BFM'14',CTY'10,Va'11}.\\
Our results extend those for parabolic equations with static boundary conditions in \cite{IY'98} to the dynamic boundary condition case.

For applications, if we limit ourselves to the particular but important case where the source terms in \eqref{eq1to4} are given by
\begin{align}
F(t,x) &=f(x)r(t,x), \quad \text{ for all } (t,x)\in \Omega_T, \label{f1}\\
G(t,x) &=g(x)\widetilde{r}(t,x), \quad \text{ for all } (t,x)\in \Gamma_T, \label{f2}
\end{align}
uniqueness and stability results can be established as a direct consequence of our Lipschitz stability result, where the inverse source problem is to determine the couple of spacewise dependent sources $(f,g)$ by the same measurements, provided that the couple of functions $(r,\widetilde{r})$ is known and satisfying some positivity assumption. The couple of functions $(f,g)$ models one special but important case of spatial distributions of source terms arising in several fields of applications such as biology, population dynamics, chemistry, etc.

The rest of the paper is organized as follows: in Section \ref{sec2}, the well-posedness of system \eqref{eq1to4} is discussed, and a special attention is paid to the regularity results, since the Bukhgeim-Klibanov method requires some regularity on the time derivative of the solution, and then we present the Carleman estimate relevant to the system \eqref{eq1to4}. Finally, in Section \ref{sec3}, we apply the Carleman estimate to prove the Lipschitz stability result.

\section{General Framework}\label{sec2}
\subsection{Functional setting}
We denote the Lebesgue measure on $\Omega$ and the surface measure on $\Gamma$ by $\d x$ and $\d S$, respectively. We will use the following real spaces
$$\mathbb{L}^2:=L^2(\Omega, \d x)\times L^2(\Gamma, \d S), \quad \mathbb{L}^2_T:=L^2(\Omega_T)\times L^2(\Gamma_T).$$
$\mathbb{L}^2$ is a real Hilbert space with the corresponding scalar product given by
$$\langle (y,y_\Gamma),(z,z_\Gamma)\rangle_{\mathbb{L}^2} =\langle y,z\rangle_{L^2(\Omega)} +\langle y_\Gamma,z_\Gamma\rangle_{L^2(\Gamma)}.$$
Analogously to $H^k(\Omega)$ and $H^k(\Gamma)$, the usual second order Sobolev spaces over $\Omega$ and $\Gamma$, we consider $$\mathbb{H}^k:=\{(y,y_\Gamma)\in H^k(\Omega)\times H^k(\Gamma)\colon y\rvert_\Gamma =y_\Gamma \}, \text{ for } k=1,2,$$ with the standard norm induced by $H^k(\Omega)\times H^k(\Gamma)$. Recall that $\|y\|_{L^2(\Gamma)}+\|\nabla_\Gamma y\|_{L^2(\Gamma)}$ defines an equivalent norm on $H^1(\Gamma)$. Moreover, $\|y\|_{L^2(\Gamma)}+\|\Delta_\Gamma y\|_{L^2(\Gamma)}$ yields an equivalent norm on $H^2(\Gamma)$.

For the regularity of the solution we introduce the following spaces
$$\mathbb{E}_1(t_0,t_1):=H^1(t_0,t_1 ;\mathbb{L}^2) \cap L^2(t_0,t_1 ;\mathbb{H}^2) \text{  for } t_1 >t_0 \text{ in } \mathbb{R},$$
$$\mathbb{E}_2(t_0,t_1):=H^1(t_0,t_1;\mathbb{H}^2) \cap H^2(t_0,t_1;\mathbb{L}^2) \text{  for } t_1 >t_0 \text{ in } \mathbb{R}.$$
In particular,
$$\mathbb{E}_1 := \mathbb{E}_1(0,T) \;\text{ and }\; \mathbb{E}_2:= \mathbb{E}_2(0,T).$$
For the known part $(r,\widetilde{r})$ of the source term $(F,G)$ in the special form \eqref{f1}-\eqref{f2}, we use the following space
$$\mathcal{C}^{1,0}:=C^{1,0}([0,T]\times\overline{\Omega})\times C^{1,0}([0,T]\times \Gamma),$$
where $C^{1,0}([0,T]\times E)=\{y=y(t,x)\rvert \; y,\partial_t y \in C([0,T]\times E)\}$ for $E=\overline{\Omega} \text{ or } \Gamma$.

We conclude by recalling an important regularity result that we will use in the sequel. Since $\Omega$ is assumed to be of class $C^2$ and $A \in C^1(\overline{\Omega}; \mathbb{R}^{N\times N})$, the elliptic regularity states that: if $y\in H^1(\Omega)$ is such that $\dv (A \nabla y) \in L^2(\Omega)$ and the trace $y\rvert_\Gamma \in H^2(\Gamma)$, then $y \in H^2(\Omega)$, see for instance \cite[Theorem 9.3.3]{Jo'07}. A similar regularity result holds for the elliptic operator on $\Gamma$. If $u\in H^1(\Gamma)$ and $\dv_\Gamma (D\nabla_\Gamma u)\in L^2(\Gamma)$, then $u\in H^2(\Gamma)$, see e.g., \cite[Proposition 1.6]{Ta'11}.

\subsection{Well-posedness and time regularity of the solution}
In this section, we mainly borrow our terminology from \cite{Ou'04}.

The system \eqref{eq1to4} can be written in the following abstract form
\begin{numcases}{\text{(ACP)}\label{acp}}
\hspace{-0.15cm} \partial_t Y=\mathcal{A} Y+ \mathcal{F}, \quad 0<t<T, \nonumber\\
\hspace{-0.15cm} Y(0)=Y_0=(y_0, y_{0,\Gamma}), \nonumber
\end{numcases}
where $Y:=(y,y_{\Gamma})$, $\mathcal{F}=(F,G)$ and the linear operator $$\mathcal{A} \colon D(\mathcal{A}) \subset \mathbb{L}^2 \longrightarrow \mathbb{L}^2$$
given by
\begin{equation}
\mathcal{A}=\begin{pmatrix} \dv(A\nabla)-B\cdot \nabla -p & 0\\ -\partial_\nu^A & \dv_\Gamma(D\nabla_\Gamma) - \langle b, \nabla_\Gamma \rangle_{\Gamma} - q\end{pmatrix}, \quad D(\mathcal{A})=\mathbb{H}^2. \label{E21}
\end{equation}
Following \cite{MMS'17}, we introduce the densely defined bilinear form given by
\begin{align*}
\mathfrak{a}[(y,y_\Gamma),(z,z_\Gamma)]&=\bigintsss_\Omega \left[A(x)\nabla y \cdot \nabla z +  (B(x) \cdot\nabla y)z + p y z \right] \d x \\
& \qquad + \int_\Gamma \left[\langle D(x) \nabla_\Gamma y_\Gamma, \nabla_\Gamma z_\Gamma \rangle_\Gamma + \langle b(x), \nabla_\Gamma y_{\Gamma} \rangle_{\Gamma} z_\Gamma + q y_\Gamma z_\Gamma \right] \d S,
\end{align*}
with form domain $D(\mathfrak{a})=\mathbb{H}^1$ on the Hilbert space $\mathbb{L}^2$.
For  a real number $\mu$, we denote by $\mathfrak{a}+\mu$ the following bilinear form
$$(\mathfrak{a}+\mu)[(y,y_\Gamma),(z,z_\Gamma)]=\mathfrak{a}[(y,y_\Gamma),(z,z_\Gamma)] + \mu \langle (y,y_\Gamma),(z,z_\Gamma) \rangle_{\mathbb{L}^2}.$$
By virtue of \eqref{uellipA}-\eqref{uellipD} and using Cauchy-Schwarz inequality, there is a constant $\mu \in \mathbb{R}$ such that
$$\mathfrak{a}[(y,y_\Gamma),(y,y_\Gamma)]+\mu \|(y,y_\Gamma)\|_{\mathbb{L}^2}^2 \geq \frac{\beta_0}{2} \|(y,y_\Gamma)\|_{\mathbb{H}^1}^2 \quad \text{ for all } (y,y_\Gamma)\in \mathbb{H}^1.$$
Therefore, following \cite{Ou'04} one can check that the form $\mathfrak{a}+\mu$ is densely defined, accretive, continuous and closed. Then, we can associate with the form $\mathfrak{a}$ an operator $\widetilde{\mathcal{A}}$ given by
\begin{align}
D(\widetilde{\mathcal{A}})&:=\{(y,y_\Gamma) \in \mathbb{H}^1, \text{ there exists } (w,w_\Gamma) \in \mathbb{L}^2 \text{ such that } \nonumber\\
&\quad \mathfrak{a}[(y,y_\Gamma),(z,z_\Gamma)]=\langle (w,w_\Gamma), (z,z_\Gamma)\rangle_{\mathbb{L}^2}\text{ for all } (z,z_\Gamma) \in \mathbb{H}^1\}, \label{fo1}\\
\widetilde{\mathcal{A}}(y,y_\Gamma)&:=-(w,w_\Gamma) \qquad \text{ for all }(y,y_\Gamma) \in D(\widetilde{A}). \label{fo2}
\end{align}
It follows from \cite[Theorem 1.52]{Ou'04} that the operator $\widetilde{\mathcal{A}}$ generates an analytic $C_0$-semigroup on $\mathbb{L}^2$.

The generation result for the operator $\mathcal{A}$ in the case of constant diffusion matrices and without drift terms was proved in \cite[Proposition 2.6]{MMS'17}, using a Lemma by Miranville-Zelik \cite{MZ'05}. We refer to \cite{KMMR'19} for a complete study of the wellposedness in the presence of drift terms. Here we prove the result in a slightly different way based on the elliptic regularity stated in Section \ref{sec2}.
\begin{proposition}
The operator $\mathcal{A}$ generates an analytic $C_0$-semigroup $(e^{t\mathcal{A}})_{t\geq 0}$ on $\mathbb{L}^2$.\label{prop1}
\end{proposition}

\begin{proof}
It suffices to prove that $\mathcal{A}=\widetilde{\mathcal{A}}$. Let $(y,y_\Gamma) \in D(\mathcal{A})$ and $(z,z_\Gamma)\in \mathbb{H}^1$. Using integration by parts and divergence formula \eqref{sdt} we obtain
\begin{align*}
\mathfrak{a}[(y,y_\Gamma),(z,z_\Gamma)] &=\int_\Omega \left[-\dv(A(x)\nabla y) z +  (B(x) \cdot\nabla y)z + p y z\right] \d x\\
& + \int_\Gamma [-\dv_\Gamma (D(x)\nabla_\Gamma y_\Gamma) z_\Gamma + \partial_\nu^A y\, z_\Gamma + \langle b(x), \nabla_\Gamma y_{\Gamma} \rangle_{\Gamma}z_\Gamma + q y_\Gamma z_\Gamma] \d S\\
&= \langle -\mathcal{A}(y,y_\Gamma), (z,z_\Gamma)\rangle_{\mathbb{L}^2}.
\end{align*}
Then, $D(\mathcal{A}) \subseteq D(\widetilde{\mathcal{A}})$ and $\widetilde{\mathcal{A}}(y,y_\Gamma)=\mathcal{A}(y,y_\Gamma)$ for all $(y,y_\Gamma) \in D(\mathcal{A})$, i.e., the operator $\widetilde{\mathcal{A}}$ is an extension of $\mathcal{A}$.
In order to prove the converse, let $(y,y_\Gamma) \in D(\widetilde{\mathcal{A}})$. The above calculation implies that $\dv (A \nabla y) \in L^2(\Omega)$ and $\dv_\Gamma(D \nabla y_\Gamma) \in L^2(\Gamma)$. Since $y_\Gamma \in H^1(\Gamma)$, the elliptic regularity on $\Gamma$ yields that $y_\Gamma \in H^2(\Gamma)$, and by the same argument for the operator on $\Omega$, we have $y\in H^2(\Omega)$. Hence, $D(\widetilde{\mathcal{A}}) \subseteq D(\mathcal{A})$. Finally, $\mathcal{A}=\widetilde{\mathcal{A}}$ and $\mathcal{A}$ generates an analytic $C_0$-semigroup on $\mathbb{L}^2$.
\end{proof}
In the sequel, we adopt the following notions of solutions.
\begin{definition}
Let $(F,G) \in \mathbb{L}^2_T$ and $Y_0:=(y_0, y_{0,\Gamma})\in \mathbb{L}^2$.
\begin{enumerate}[label=(\alph*),leftmargin=*]
\item A strong solution of \eqref{eq1to4} is a function $Y:=(y, y_{\Gamma}) \in \mathbb{E}_1$ fulfilling \eqref{eq1to4} in $L^2(0,T; \mathbb{L}^2)$.
\item A mild solution of \eqref{eq1to4} is a function $Y:=(y, y_{\Gamma}) \in C([0,T]; \mathbb{L}^2)$ satisfying, for $t\in [0,T]$,
$$Y(t,\cdot)=e^{t\mathcal{A}} Y_0 + \int_0^t e^{(t-\tau)\mathcal{A}} [F(\tau, \cdot), G(\tau, \cdot)] \,\d\tau.$$
\end{enumerate}
\end{definition}
Since $\mathcal{A}$ generates an analytic $C_0$-semigroup on $\mathbb{L}^2$, the following regularity result holds. See for instance Theorem 3.1 and Proposition 3.8 in \cite{Ben'07}.
\begin{proposition}
Let $F\in L^2(\Omega_T)$ and $G\in L^2(\Gamma_T)$.
\begin{enumerate}[label=(\roman*),leftmargin=*]
\item  For all $Y_0:=(y_0,y_{0,\Gamma})\in \mathbb{H}^1$, there exists a unique strong solution of \eqref{eq1to4} such that $$Y:=(y,y_{\Gamma})\in \mathbb{E}_1 :=H^1(0,T;\mathbb{H}^2) \cap L^2(0,T;\mathbb{L}^2) .$$
\item For all $Y_0:=(y_0,y_{0,\Gamma})\in \mathbb{L}^2$, there exists a unique mild solution of \eqref{eq1to4} $Y:=(y,y_{\Gamma})\in C([0,T]; \mathbb{L}^2)$ such that for all $\tau \in (0,T)$, $$Y\in \mathbb{E}_1(\tau, T):=H^1(\tau,T;\mathbb{H}^2) \cap L^2(\tau,T;\mathbb{L}^2).$$
\end{enumerate}
Moreover, if $\mathcal{F}=(F,G)\in H^1(0,T; \mathbb{L}^2)$, then for all $\tau \in (0,T)$, we have
$$Y\in \mathbb{E}_2(\tau, T):=H^1(\tau,T;\mathbb{H}^2) \cap H^2(\tau,T;\mathbb{L}^2).$$
\label{prop2}
\end{proposition}

\subsection{Carleman estimate}
To state and prove our Carleman estimate, we need a weight function with special properties. The existence of such function is proved in \cite{FI'96}.
\begin{lemma}
Let $\omega' \Subset \Omega$ be a nonempty open subset. Then there is a function $\eta^0\in C^2(\overline{\Omega})$ such that       
\begin{equation*}
\eta^0> 0  \quad \text{ in } \Omega, \qquad \eta^0=0 \quad \text{ on } \Gamma, \qquad |\nabla\eta^0|  > 0 \quad\text{ in } \overline{\Omega\backslash\omega'}.
\end{equation*}
Moreover, the identity $|\nabla\eta^0|^2=|\nabla_\Gamma \eta^0|^2+|\partial_\nu \eta^0|^2$ on $\Gamma$ implies
\begin{equation}
\nabla_\Gamma \eta^0 =0, \qquad |\nabla \eta^0|=|\partial_\nu \eta^0|, \qquad \partial_\nu \eta^0 \leq -c <0 \quad\text{ on } \Gamma \label{eqlm1}
\end{equation}
for some constant $c>0$.
\label{lm1}
\end{lemma}

\begin{remark}
The identity $(\partial_\nu \psi)^2=|\nabla \psi|^2-|\nabla_\Gamma \psi|^2$ and the property $\partial_\nu \eta^0 <-c<0$ have played important roles in the proof of Carleman estimate with standard Laplacians in \cite{MMS'17}. Since we deal with general elliptic second order operators, we need similar properties with the conormal derivative instead of the normal derivative. This is the purpose of the following lemma.
\end{remark}

\begin{lemma} Let $\psi$ be any smooth function.
\begin{enumerate}[label=(\roman*)]
\item The following identity holds
\begin{equation}
(\partial_\nu^A \psi)^2 -(A\nabla_\Gamma \psi\cdot \nu)^2=|A^{\frac{1}{2}}\nu|^2 \left(|A^{\frac{1}{2}}\nabla \psi|^2-|A^{\frac{1}{2}}\nabla_\Gamma \psi|^2 \right). \label{eqconormal}
\end{equation}
\item Let $c$ be the same constant in \eqref{eqlm1}. Then
\begin{equation}
\partial_\nu^A \eta^0 \leq \beta_0 \partial_\nu \eta^0 \leq -c \beta_0 <0. \label{ceqlm1}
\end{equation}
\end{enumerate}
\end{lemma}

\begin{proof}
$(i)$ Using the identity \eqref{eqtgrad} we obtain
\begin{align}
A\nabla \psi &=A\nabla_\Gamma \psi +(\partial_\nu \psi)A\nu \label{gd1}.\\
\nabla \psi &=\nabla_\Gamma \psi +(\partial_\nu \psi)\nu. \label{gd2}
\end{align}
Composing \eqref{gd1} by $\nu$ yields the following identity
\begin{equation*}
\partial_\nu^A \psi= A\nabla_\Gamma \psi\cdot \nu + (\partial_\nu \psi) (A\nu\cdot \nu).
\end{equation*}
By taking the scalar product of \eqref{gd1} and \eqref{gd2} with multiplication of the resulting identity by $(A\nu\cdot \nu)$ we infer that
$$(A\nu\cdot \nu) (A\nabla\psi\cdot \nabla\psi -A\nabla_\Gamma\psi\cdot \nabla_\Gamma\psi)=(\partial_\nu \psi)^2 (A\nu\cdot \nu)^2 +  2 (\partial_\nu \psi) (A\nu\cdot \nu) A\nabla_\Gamma \psi\cdot \nu,$$
where we used the symmetry of $A$. Completing the square yields the result.\\
$(ii)$ Since $\eta^0 \rvert_\Gamma=0$, we obtain
\begin{equation}
\nabla \eta^0 =(\partial_\nu \eta^0) \nu\qquad \text{ on } \Gamma. \label{eqgn}
\end{equation}
Then, $\partial_\nu^A \eta^0 = (A\nu\cdot\nu)\partial_\nu \eta^0$. Hence, by virtue of \eqref{uellipA} and \eqref{eqlm1},  we have
\begin{equation*}
\partial_\nu^A \eta^0 \leq \beta_0 \partial_\nu \eta^0 \leq -c \beta_0 <0.
\end{equation*}
\end{proof}

\begin{remark}
In the isotropic case, i.e., $A(x)=I$, since $\nabla_\Gamma \psi\cdot \nu=0$, the identity \eqref{eqconormal} is simply the same as $(\partial_\nu \psi)^2=|\nabla \psi|^2 - |\nabla_{\Gamma} \psi|^2$.
\end{remark}
We introduce the following weight functions
\begin{equation}
\alpha(t,x)=\frac{e^{2\lambda \|\eta^0\|_\infty}- e^{\lambda \eta^0(x)}}{t(T -t)} \quad \text{ and } \quad \xi(t,x)=\frac{e^{\lambda \eta^0(x)}}{t(T -t)} \label{w1}
\end{equation}
for all $(t,x)\in \overline{\Omega}_T$, and $\lambda \geq 1$ is a parameter (to fix later) which depends only on $\Omega$ and $\omega$.
Note that $\alpha$ and $\xi$ are of class $C^2$, strictly positive on $\overline{\Omega}_T$ and blow up as $t\to 0$ and as $t\to T$, and we have $$|\partial_t \alpha| \leq CT\xi^2 \quad\text{ and }\quad \xi \leq T^2 \xi^2.$$
Furthermore,
\begin{equation}
\nabla_\Gamma \alpha =0 \quad \text{ and } \quad \nabla_\Gamma \xi =0 \quad \text{ on }\Gamma. \label{2eq3.2}
\end{equation}
We notice that, for fixed $x\in \Omega$, $\alpha(\cdot, x)$ attains the minimum in $(0,T)$ at $\displaystyle \frac{T}{2}$.\\ Consider
$$L z=\partial_t z -\dv (A(x) \nabla z) + B(x)\cdot \nabla z + p(x)z, \qquad (t,x)\in \Omega_T,$$
and
$$L_\Gamma z_\Gamma=\partial_t z_\Gamma -\dv_\Gamma (D(x)\nabla_\Gamma z_\Gamma) + \partial_\nu^A z + \langle b(x), \nabla_\Gamma z_{\Gamma} \rangle_{\Gamma} +q(x)z_{\Gamma}, \qquad (t,x)\in \Gamma_T.$$
The following lemma is the key tool to prove the main result on global Lipschitz stability in our inverse source problem.
\begin{lemma}[Carleman estimate]
Let $T>0$, $\omega\Subset \Omega$ be nonempty and open subset. Consider $\eta^0$, $\alpha$ and $\xi$ as above with respect to a nonempty open set $\omega' \Subset \omega$. Then there are three positive constants $\lambda_1,s_1 \geq 1$ and $C>0$ such that, for any $\lambda\geq \lambda_1$ and $s\geq s_1$, the following inequality holds
\begin{small}
\begin{align}
& \bigintssss_{\Omega_T} \left(\frac{1}{s\xi} \left(|\partial_t z|^2 + |\dv(A\nabla z)|^2 \right)+ s\lambda^2 \xi |\nabla z|^2 + s^3\lambda^4\xi^3 |z|^2 \right)e^{-2s\alpha} \,\d x\,\d t \quad +  \nonumber\\
& \int_{\Gamma_T} \left(\frac{1}{s\xi} (|\partial_t z_\Gamma|^2 +|\dv(D\nabla_\Gamma z_\Gamma)|^2)+s\lambda \xi |\nabla_\Gamma z_\Gamma|^2 + s^3\lambda^3\xi^3 |z_\Gamma|^2 + s\lambda \xi |\partial_\nu^A z|^2\right)e^{-2s\alpha} \,\d S\,\d t \nonumber\\
& \quad\leq   C s^3\lambda^4\int_{\omega_T} e^{-2s\alpha} \xi^3 |z|^2 \,\d x\,\d t +C \int_{\Omega_T}e^{-2s\alpha}|Lz|^2 \,\d x\,\d t + C \int_{\Gamma_T}e^{-2s\alpha}|L_\Gamma z_\Gamma|^2 \,\d S\,\d t\label{car1}
\end{align}
\end{small}
for all $(z,z_\Gamma)\in \mathbb{E}_1$. Given $K>0$, the constant $C=C(K)$ can be chosen independently of all potentials $p$ and $q$ such that $\|p\|_\infty, \|q\|_\infty \leq K$.\label{lm2}
\end{lemma}
This Carleman estimate extends the one obtained in  \cite[Lemma 3.2]{MMS'17}. Since the proof is slightly different, we only have to revisit some terms by expanding the computation.
\begin{proof}
It suffices to prove the inequality \eqref{car1} for
$$L_0 z=\partial_t z -\dv (A(x) \nabla z) \qquad \text{ and } \qquad L_{0,\Gamma} z_\Gamma=\partial_t z_\Gamma -\dv_\Gamma (D(x)\nabla_\Gamma z_\Gamma) + \partial_\nu^A z,$$
since lower order terms with bounded coefficients do not influence the Carleman estimate. Taking into account \eqref{symA} and \eqref{uellipA},  we denote $A_0 :=\sup\limits_{x\in \overline{\Omega}} \sup\limits_{|\zeta|=1} \langle A(x)\zeta, \zeta \rangle.$ Then,  we have
\begin{equation}
\beta_0 |\zeta|^2 \leq \langle A(x)\zeta, \zeta\rangle \leq A_0 |\zeta|^2, \qquad\qquad x\in \overline{\Omega}, \;\zeta \in \mathbb{R}^N . \label{uellipA1}
\end{equation}
\textbf{Step 1. Conjugate operators.}\\
Let $z\in C^\infty([0,T]\times \overline{\Omega})$, $\lambda\geq \lambda_1 \geq 1$ and $s\geq s_1 \geq 1$ be given. Set
$$\psi := e^{-s\alpha} z, \qquad f:=e^{-s\alpha} L_0 z, \qquad g:= e^{-s\alpha} L_{0,\Gamma} z_\Gamma, \qquad \sigma:=A(\cdot)\nabla \eta^0 \cdot \nabla\eta^0.$$
By definition and by \eqref{uellipA} there exists a positive constant $C_1>0$ such that
\begin{equation}
\beta_0 |\nabla \eta^0|^2 \leq \sigma(x) \leq  C_1, \quad x\in \overline{\Omega}.\label{sigmaineq}
\end{equation}
The corresponding conjugate operators of $L_0$ and $L_{0,\Gamma}$ are given by
$$M\psi :=e^{-s\alpha} L_0(e^{s\alpha} \psi)=e^{-s\alpha} L_0 z, \qquad N\psi_\Gamma :=e^{-s\alpha} L_{0,\Gamma}(e^{s\alpha} \psi_\Gamma)=e^{-s\alpha} L_{0,\Gamma} z_\Gamma.$$
For the sake of simplicity, we will write $z$ and $\psi$ instead of $z_\Gamma$ and $\psi_\Gamma$ on $\Gamma_T$.\\
First, we determine the problem fulfilled by $\psi$ by expanding the spatial derivatives of $\alpha$ and using the symmetry of $A$. We have
\begin{align}
\nabla\alpha &= -\nabla \xi=-\lambda \xi \nabla \eta^0, \label{2eq3.3}\\
\dv(A(x)\nabla \alpha) &= -\lambda^2 \xi \sigma -\lambda \xi \dv(A(x)\nabla\eta^0), \nonumber\\
\partial_t \psi &= e^{-s\alpha} \partial_t z-s\psi \partial_t \alpha, \nonumber\\
\nabla \psi &= e^{-s\alpha} \nabla z +s\lambda \psi \xi \nabla \eta^0 , \label{2eq3.4}\\
\dv(A(x)\nabla \psi) &= e^{-s\alpha} \dv(A(x)\nabla z) + 2s \lambda \xi  A(x)\nabla \eta^0\cdot \nabla\psi - s^2\lambda^2 \xi^2 \psi \sigma \nonumber\\
& \qquad + s\lambda \xi \psi \dv(A(x)\nabla\eta^0)+ s\lambda^2 \xi \psi \sigma .\nonumber
\end{align}
Regrouping the previous formulae we obtain the following evolution equation
\begin{align}
\partial_t \psi -\dv(A(x)\nabla \psi) &= f - 2s\lambda \xi A(x)\nabla \eta^0\cdot \nabla\psi -s\lambda^2 \xi \psi \sigma  \label{2eq3.5}\\
& \quad + s^2\lambda^2 \xi^2 \psi \sigma -s\lambda \xi \psi \dv(A(x)\nabla \eta^0) - s\psi \partial_t \alpha. \nonumber
\end{align}
Similarly, on $\Gamma_T$ we obtain
\begin{equation}
\partial_t \psi -\dv_\Gamma(D(x)\nabla_\Gamma \psi) + \partial_\nu^A  \psi =g-s\psi \partial_t \alpha + s\lambda \psi \xi \partial_\nu^A \eta^0. \label{2eq3.6}
\end{equation}
Extending the corresponding decomposition in \cite{MMS'17}, we rewrite the equations \eqref{2eq3.5} and \eqref{2eq3.6} as
\begin{equation}
M_1 \psi + M_2 \psi =\tilde{f}\; \text{ in } \Omega_T, \qquad N_1 \psi + N_2 \psi =g\; \text{ on } \Gamma_T, \label{2eq3.7}
\end{equation}
where
\begin{align*}
M_1 \psi &= 2 s\lambda^2 \psi \xi \sigma + 2s \lambda \xi A(x)\nabla \eta^0\cdot \nabla\psi +\partial_t \psi =M_{1,1}\psi + M_{1,2}\psi + M_{1,3}\psi,\\
M_2 \psi &= -s^2\lambda^2 \xi^2 \psi \sigma -\dv(A(x)\nabla \psi) + s\psi \partial_t \alpha =M_{2,1}\psi + M_{2,2}\psi + M_{2,3}\psi,\\
N_1 \psi &= \partial_t \psi - s\lambda \psi \xi \partial_\nu^A \eta^0 = N_{1,1}\psi + N_{1,2}\psi,\\
N_2 \psi &= -\dv_\Gamma(D(x)\nabla_\Gamma \psi) + s\psi \partial_t \alpha + \partial_\nu^A \psi = N_{2,1}\psi + N_{2,2}\psi + N_{2,3}\psi,\\
\tilde{f} &= f - s\lambda \xi \psi \dv(A(x)\nabla \eta^0) + s\lambda^2 \psi \xi \sigma.
\end{align*}
By taking $\|\cdot\|_{L^2(\Omega_T)}^2$ and $\|\cdot\|_{L^2(\Gamma_T)}^2$ in the equations \eqref{2eq3.7} and adding the resulting identities, we obtain
\begin{align}
&\|\tilde{f}\|_{L^2(\Omega_T)}^2 + \|g\|_{L^2(\Gamma_T)}^2 = \|M_1 \psi\|_{L^2(\Omega_T)}^2 + \|M_2 \psi\|_{L^2(\Omega_T)}^2 + \|N_1 \psi\|_{L^2(\Gamma_T)}^2 \label{2eq3.8}\\
& \quad + \|N_2 \psi\|_{L^2(\Gamma_T)}^2 + 2 \sum_{i,j=1}^N \langle M_{1,i}\psi, M_{2,j}\psi\rangle_{L^2(\Omega_T)} + 2 \sum_{i,j=1}^N \langle N_{1,i}\psi, N_{2,j}\psi \rangle_{L^2(\Gamma_T)} \nonumber.
\end{align}
\textbf{Step 2. Estimating the mixed terms from below}. We will use the following estimates on $\overline{\Omega}$ in the sequel,
\begin{equation}
|\nabla \alpha| \leq C \lambda \xi, \qquad |\partial_t \alpha| \leq C \xi^2, \qquad |\partial_t \xi| \leq C \xi^2. \label{2eq3.9}
\end{equation}

\textbf{Step 2a.} The first term is negative
\begin{equation*}
\langle M_{1,1}\psi, M_{2,1}\psi\rangle_{L^2(\Omega_T)} =-2 s^3\lambda^4 \int_{\Omega_T} \sigma^2 \xi^3 \psi^2 \,\d x\,\d t.
\end{equation*}
By integration by parts and \eqref{2eq3.3},  we obtain
\begin{align*}
& \langle M_{1,2}\psi, M_{2,1}\psi\rangle_{L^2(\Omega_T)} =-s^3\lambda^3 \int_{\Omega_T} \xi^3 \sigma  A(x)\nabla \eta^0  \cdot \nabla (\psi^2) \,\d x\,\d t.\\
& \quad = s^3 \lambda^3 \int_{\Omega_T} \dv(\xi^3 \sigma  A(x) \nabla \eta^0)\psi^2 \,\d x\,\d t -s^3\lambda^3 \int_{\Gamma_T} \xi^3 \sigma \partial_\nu^A \eta^0 \psi^2 \,\d S\,\d t\\
& \quad = 3 s^3 \lambda^4 \int_{\Omega_T} \sigma^2 \xi^3 \psi^2 \,\d x\,\d t + s^3 \lambda^3 \int_{\Omega_T}  \xi^3 (\nabla \sigma \cdot A(x) \nabla \eta^0) \psi^2 \,\d x\,\d t\\
& \qquad + s^3 \lambda^3 \int_{\Omega_T}  \xi^3  \sigma  \dv(A(x)\nabla \eta^0) \psi^2 \,\d x\,\d t -s^3\lambda^3 \int_{\Gamma_T} \xi^3 \sigma \partial_\nu^A \eta^0 \psi^2 \,\d S\,\d t.
\end{align*}
Using the fact that $\nabla\eta^0 \neq 0$ on $\overline{\Omega \setminus \omega'}$, \eqref{sigmaineq} and \eqref{ceqlm1},  we obtain 
\begin{align*}
&\langle M_{1,1}\psi, M_{2,1}\psi\rangle_{L^2(\Omega_T)} + \langle M_{1,2}\psi, M_{2,1}\psi\rangle_{L^2(\Omega_T)} \\
& \geq C s^3 \lambda^4 \int_{\Omega_T} \xi^3 \psi^2 \,\d x\,\d t -C s^3 \lambda^4 \int_{(0,T)\times \omega'} \xi^3 \psi^2 \,\d x\,\d t -Cs^3\lambda^3 \int_{\Gamma_T} \xi^3 \sigma \partial_\nu \eta^0 \psi^2 \,\d S\,\d t\\
& \geq C s^3 \lambda^4 \int_{\Omega_T} \xi^3 \psi^2 \,\d x\,\d t -C s^3 \lambda^4 \int_{(0,T)\times \omega'} \xi^3 \psi^2 \,\d x\,\d t + Cs^3\lambda^3 \int_{\Gamma_T} \xi^3 \psi^2 \,\d S\,\d t.
\end{align*}
After integrating by parts in time and using \eqref{sigmaineq} with \eqref{2eq3.9} we obtain
\begin{align*}
\langle M_{1,3}\psi, M_{2,1}\psi\rangle_{L^2(\Omega_T)} &=-\frac{1}{2} s^2 \lambda^2 \int_{\Omega_T} \sigma \xi^2 \partial_t (\psi^2) \,\d x\,\d t = s^2 \lambda^2 \int_{\Omega_T} \sigma \partial_t \xi \xi \psi^2 \,\d x\,\d t\\
& \geq -C s^2 \lambda^2 \int_{\Omega_T} \xi^3 \psi^2 \,\d x\,\d t,
\end{align*}
since $\psi$ vanishes at $t=0$ and $t=T$.

\textbf{Step 2b.} Integration by parts and \eqref{uellipA} yield
\begin{align}
& \langle M_{1,1}\psi, M_{2,2}\psi\rangle_{L^2(\Omega_T)}=-2 s\lambda^2 \int_{\Omega_T} \sigma \xi \psi \dv(A(x)\nabla\psi) \,\d x\,\d t \nonumber\\
& \; = 2 s\lambda^2 \int_{\Omega_T} \nabla (\sigma \xi \psi) \cdot A(x)\nabla \psi \,\d x\,\d t - 2s\lambda^2 \int_{\Gamma_T} \sigma \xi \psi \partial_\nu^A \psi \,\d S\,\d t \nonumber\\
& \; = 2 s\lambda^2 \int_{\Omega_T} \sigma \xi A(x)\nabla \psi \cdot\nabla \psi \,\d x\,\d t + 2 s\lambda^2 \int_{\Omega_T}  \xi \psi \nabla\sigma \cdot A(x)\nabla \psi \,\d x\,\d t \nonumber\\
& \qquad + 2 s\lambda^3 \int_{\Omega_T}  \sigma \xi \psi \nabla\eta^0 \cdot A(x)\nabla \psi \,\d x\,\d t - 2s\lambda^2 \int_{\Gamma_T} \sigma \xi \psi \partial_\nu^A \psi \,\d S\,\d t \nonumber\\
& \; \geq 2 s\lambda^2 \int_{\Omega_T} \sigma \xi A(x)\nabla \psi \cdot\nabla \psi \,\d x\,\d t -C s^2\lambda^4 \int_{\Omega_T} \xi^2 \psi^2 \,\d x\,\d t \nonumber\\
& \qquad - C \int_{\Omega_T} (s\xi +\lambda^2) |\nabla\psi|^2 \,\d x\,\d t - 2s\lambda^2 \int_{\Gamma_T} \sigma \xi \psi \partial_\nu^A \psi \,\d S\,\d t, \label{eqj0}
\end{align}
where we employed Cauchy-Schwarz inequality for the terms in the middle as in \cite{MMS'17} and \eqref{sigmaineq}. Using integration by parts and $\partial_i (a_{ij} \partial_j \psi)=a_{ij} \partial_i \partial_j \psi + \partial_i (a_{ij}) \partial_j \psi$, with help of \eqref{2eq3.3} the next addend becomes
\begin{align*}
& \langle M_{1,2}\psi, M_{2,2}\psi\rangle_{L^2(\Omega_T)} =-2s\lambda \int_{\Omega_T} \xi (\nabla \eta^0 \cdot A(x)\nabla \psi) \dv(A(x)\nabla \psi) \,\d x\,\d t\\
& \quad = -2s\lambda \int_{\Omega_T} \sum_{i,j=1}^N \sum_{k,l=1}^N \xi a_{ij} a_{kl} (\partial_k \eta^0) (\partial_l \psi) (\partial_i \partial_j \psi) \,\d x\,\d t \\
& \quad \qquad -2s\lambda \int_{\Omega_T} \xi \sum_{i,j=1}^N \sum_{k,l=1}^N a_{kl} (\partial_k \eta^0) (\partial_l \psi) \partial_i (a_{ij}) \partial_j \psi \,\d x\,\d t\\
& \quad = 2s\lambda^2 \int_{\Omega_T} \sum_{i,j=1}^N \sum_{k,l=1}^N (\partial_i \eta^0)\xi a_{ij} a_{kl} (\partial_k \eta^0) (\partial_l \psi) (\partial_j \psi) \,\d x\,\d t\\
& \qquad +  2s\lambda \int_{\Omega_T} \sum_{i,j=1}^N \sum_{k,l=1}^N \xi \partial_i (a_{ij} a_{kl} \partial_k \eta^0) (\partial_l \psi) (\partial_j \psi) \,\d x\,\d t\\
& \qquad +  2s\lambda \int_{\Omega_T} \sum_{i,j=1}^N \sum_{k,l=1}^N \xi  a_{ij} a_{kl} \partial_k \eta^0 (\partial_i \partial_l \psi) (\partial_j \psi) \,\d x\,\d t\\
& \qquad - 2s\lambda \int_{\Gamma_T} \sum_{k,l=1}^N \xi  a_{kl} (\partial_k \eta^0) (\partial_l \psi) \partial_\nu^A \psi \,\d S\,\d t\\
& \qquad -2s\lambda \int_{\Omega_T} \xi \sum_{i,j=1}^N \sum_{k,l=1}^N a_{kl} (\partial_k \eta^0) (\partial_l \psi) \partial_i (a_{ij}) \partial_j \psi \,\d x\,\d t\\
&= 2s\lambda^2 \bigintsss_{\Omega_T} \xi \left|\nabla \eta^0 \cdot A(x)\nabla \psi\right|^2 \,\d x\,\d t\\
& \qquad +  2s\lambda \int_{\Omega_T} \sum_{i,j=1}^N \sum_{k,l=1}^N \xi \partial_i (a_{ij} a_{kl} \partial_k \eta^0) (\partial_l \psi) (\partial_j \psi) \,\d x\,\d t\\
& \qquad +  2s\lambda \int_{\Omega_T} \sum_{i,j=1}^N \sum_{k,l=1}^N \xi  a_{ij} a_{kl} \partial_k \eta^0 (\partial_i \partial_l \psi) (\partial_j \psi) \,\d x\,\d t \\
& \qquad - 2s\lambda \int_{\Gamma_T} \xi (\partial_\nu \eta^0) (\partial_\nu^A \psi)^2 \,\d S\,\d t\\
& \qquad -2s\lambda \int_{\Omega_T} \xi \sum_{i,j=1}^N \sum_{k,l=1}^N a_{kl} (\partial_k \eta^0) (\partial_l \psi) \partial_i (a_{ij}) \partial_j \psi \,\d x\,\d t\\
& = D_1+D_2+D_3+D_4 +D_5.
\end{align*}
Observe that the first term $D_1$ is nonnegative. Similarly to previous integration by parts we obtain
\begin{align*}
D_3 & = s\lambda \int_{\Omega_T} \sum_{i,j=1}^N \sum_{k,l=1}^N \xi  a_{ij} a_{kl} (\partial_k \eta^0) \partial_l [(\partial_i \psi)(\partial_j \psi)] \,\d x\,\d t\\
& = - s\lambda^2 \int_{\Omega_T} \xi \sigma \sum_{i,j=1}^N  a_{ij} (\partial_i \psi)(\partial_j \psi) \,\d x\,\d t \\
& \quad - s\lambda \int_{\Omega_T} \xi \sum_{i,j=1}^N \sum_{k,l=1}^N  \partial_l (a_{ij} a_{kl}\partial_k \eta^0) (\partial_i \psi)(\partial_j \psi) \,\d x\,\d t\\
& \quad + s\lambda \int_{\Gamma_T} \xi \partial_\nu^A \eta^0 \sum_{i,j=1}^N  a_{ij} (\partial_i \psi)(\partial_j \psi) \,\d S\,\d t\\
&= - s\lambda^2 \int_{\Omega_T} \sigma \xi A(x)\nabla \psi\cdot \nabla \psi \,\d x\,\d t\\
& \quad - s\lambda \int_{\Omega_T} \xi \sum_{i,j=1}^N \sum_{k,l=1}^N  \partial_l (a_{ij} a_{kl}\partial_k \eta^0) (\partial_i \psi)(\partial_j \psi) \,\d x\,\d t\\
& \quad + s\lambda \int_{\Gamma_T} \xi (\partial_\nu \eta^0) (A(x)\nu\cdot \nu)(A(x)\nabla \psi \cdot \nabla \psi) \,\d S\,\d t,
\end{align*}
where we employed $\partial_\nu^A \eta^0 =(\partial_\nu \eta^0) (A\nu\cdot \nu)$, since $\eta^0\rvert_\Gamma=0$. Then, using the symmetry of $A$ and \eqref{eqgn} we obtain 
\begin{align}
&\langle M_{1,2}\psi, M_{2,2}\psi\rangle_{L^2(\Omega_T)} \nonumber\\
& \geq - 2s\lambda \int_{\Gamma_T} \xi (\partial_\nu \eta^0) (\partial_\nu^A \psi)^2 \,\d S\,\d t + s\lambda \int_{\Gamma_T} \xi (\partial_\nu \eta^0) (A(x)\nu\cdot\nu)(A(x)\nabla \psi \cdot \nabla\psi) \,\d S\,\d t \nonumber\\
& \qquad  - C s\lambda \int_{\Omega_T} \xi |\nabla \psi|^2 \,\d x\,\d t - s\lambda^2 \int_{\Omega_T} \sigma \xi A(x)\nabla \psi \cdot \nabla\psi \,\d x\,\d t \nonumber\\
& \geq \underbrace{- s\lambda \int_{\Gamma_T} \xi (\partial_\nu \eta^0) (\partial_\nu^A\psi)^2 \,\d S\,\d t + s\lambda\int_{\Gamma_T} \xi (\partial_\nu \eta^0) (A(x)\nu\cdot\nu)(A(x)\nabla \psi \cdot \nabla\psi) \,\d S\,\d t}_{J} \nonumber\\
& \qquad - C s\lambda \int_{\Omega_T} \xi |\nabla \psi|^2 \,\d x\,\d t -s\lambda^2 \int_{\Omega_T} \sigma \xi A(x)\nabla \psi \cdot \nabla\psi \,\d x\,\d t \nonumber\\
& \qquad  - s\lambda \int_{\Gamma_T} \xi (\partial_\nu \eta^0) (\partial_\nu^A\psi)^2 \,\d S\,\d t .\label{eqj1}
\end{align}
Next we estimate $J$ with help of \eqref{eqconormal} and \eqref{uellipA1} as follows
\begin{align*}
J &= s\lambda \int_{\Gamma_T} \xi (\partial_\nu \eta^0) [-(\partial_\nu^A\psi)^2 + |A^{\frac{1}{2}}\nu|^2 |A^{\frac{1}{2}}\nabla \psi|^2] \,\d S\,\d t\\
&= s\lambda \int_{\Gamma_T} \xi (\partial_\nu \eta^0) [|A^{\frac{1}{2}}\nu|^2 |A^{\frac{1}{2}}\nabla_\Gamma \psi|^2 - (A\nabla_\Gamma \psi \cdot \nu)^2] \,\d S\,\d t\\
& \geq s\lambda \int_{\Gamma_T} \xi (\partial_\nu \eta^0) |A^{\frac{1}{2}}\nu|^2 |A^{\frac{1}{2}}\nabla_\Gamma \psi|^2 \,\d S\,\d t\\
& \geq C s\lambda \int_{\Gamma_T} \xi (\partial_\nu \eta^0) |\nabla_\Gamma \psi|^2 \,\d S\,\d t.
\end{align*}
Combining this with \eqref{eqj1}, we derive
\begin{align*}
\langle M_{1,2}\psi, M_{2,2}\psi\rangle_{L^2(\Omega_T)} &\geq - s\lambda \int_{\Gamma_T} \xi (\partial_\nu \eta^0) (\partial_\nu^A\psi)^2 \,\d S\,\d t \\
& \quad + Cs\lambda \int_{\Gamma_T} \xi (\partial_\nu \eta^0) |\nabla_\Gamma \psi|^2 \,\d S\,\d t \\
& \quad  - C s\lambda \int_{\Omega_T} \xi |\nabla \psi|^2 \,\d x\,\d t -s\lambda^2 \int_{\Omega_T} \sigma \xi A(x)\nabla \psi \cdot \nabla\psi \,\d x\,\d t. 
\end{align*}
The last term  cancels with the one from \eqref{eqj0}.
Integration by parts once again, we obtain
\begin{align}
\langle M_{1,3}\psi, M_{2,2}\psi\rangle_{L^2(\Omega_T)} &= -\int_{\Omega_T} \partial_t \psi\, \dv(A(x)\nabla\psi) \,\d x\,\d t \nonumber\\
& =\int_{\Omega_T} A(x)\nabla\psi \cdot \partial_t (\nabla\psi) -\int_{\Gamma_T} \partial_t \psi \partial_\nu^A \psi \,\d S\,\d t \nonumber\\
& = \frac{1}{2} \int_{\Omega_T}  \frac{\mathrm{d}}{\d t} (A(x)\nabla \psi \cdot \nabla \psi) \,\d x\, \d t -\int_{\Gamma_T} \partial_t \psi \partial_\nu^A \psi \,\d S\,\d t \nonumber\\
& = -\int_{\Gamma_T} \partial_t \psi \partial_\nu^A \psi \,\d S\,\d t, \label{eqtg}
\end{align}
where we used the symmetry of $A$ and the fact that $\nabla \psi$ vanishes at $t=0$ and $t=T$.

\textbf{Step 2c.} Using \eqref{2eq3.9}, we estimate
\begin{align*}
\langle M_{1,1}\psi, M_{2,3}\psi\rangle_{L^2(\Omega_T)} =2 s^2\lambda^2\int_{\Omega_T} \sigma \xi (\partial_t\alpha) \psi^2\,\d x\,\d t\geq - C s^2\lambda^2\int_{\Omega_T} \xi^3\psi^2 \,\d x\,\d t.
\end{align*}
Integration by parts, \eqref{2eq3.3} and \eqref{2eq3.9} imply
\begin{align*}
&\langle M_{1,2}\psi, M_{2,3}\psi\rangle_{L^2(\Omega_T)} = s^2\lambda \int_{\Omega_T} (\partial_t\alpha)\xi A(x)\nabla\eta^0 \cdot \nabla(\psi^2) \,\d x\,\d t\\
& \quad =  s^2\lambda \int_{\Gamma_T} (\partial_t \alpha) \xi \partial_\nu^A \eta^0 \, \psi^2 \,\d S\,\d t - s^2\lambda \int_{\Omega_T} \dv(\xi \partial_t\alpha A(x)\nabla\eta^0)\psi^2 \,\d x\,\d t\\
& \quad = s^2\lambda \int_{\Gamma_T} (\partial_t \alpha) \xi \partial_\nu^A \eta^0 \, \psi^2 \,\d S\,\d t - s^2\lambda \int_{\Omega_T} \nabla(\partial_t\alpha) \cdot A(x)\nabla \eta^0 \xi\psi^2 \,\d x\,\d t \\
& \qquad - s^2\lambda^2 \int_{\Omega_T} (\partial_t\alpha)\xi \sigma \psi^2 \,\d x\,\d t - s^2 \lambda \int_{\Omega_T} (\partial_t\alpha)\xi \dv(A(x)\nabla \eta^0) \psi^2 \,\d x\,\d t\\
& \quad \geq -C s^2 \lambda \int_{\Gamma_T} \xi^3 \psi^2 \,\d S\,\d t -C s^2\lambda^2 \int_{\Omega_T} \xi^3 \psi^2 \,\d x\,\d t.
\end{align*}
Since $\psi(0)=\psi(T)=0$ and $|\partial_t^2 \alpha|\leq C\xi^3$, integration by parts yields
\begin{align}
\langle M_{1,3}\psi, M_{2,3}\psi\rangle_{L^2(\Omega_T)} &=\frac{s}{2}\int_{\Omega_T} \partial_t\alpha \partial_t(\psi^2)\,\d x\,\d t= -\frac{s}{2}\int_{\Omega_T} \partial^2_t\alpha\, \psi^2 \,\d x\,\d t \label{2eq3.12}\\
&\geq  -C s\int_{\Omega_T} \xi^3 \psi^2 \,\d x\,\d t,\nonumber
\end{align}

\textbf{Step 2d. Estimating boundary terms.} For the boundary terms $N_1$ and $N_2$, we will use the divergence formula \eqref{sdt}.
Using \eqref{sdt},  we have
\begin{align*}
\langle N_{1,1}\psi, N_{2,1}\psi\rangle_{L^2(\Gamma_T)}&=-\int_{\Gamma_T} \dv_\Gamma (D(x)\nabla_\Gamma \psi) \partial_t \psi \,\,\d S\,\d t \\
&=\int_{\Gamma_T} \langle D(x)\nabla_\Gamma \psi, \partial_t (\nabla_\Gamma \psi)\rangle_\Gamma \,\d S\,\d t\\
&=\frac{1}{2} \int_{\Gamma_T}  \frac{\mathrm{d}}{\d t} \langle D(x)\nabla_\Gamma \psi, \nabla_\Gamma \psi \rangle_\Gamma \,\d S\, \d t=0,
\end{align*}
by means of $\psi(0)=\psi(T)=0$. Since $\xi(t,\cdot)$ is constant on $\Gamma$, \eqref{sdt} and \eqref{uellipD} yield
\begin{align*}
&\langle N_{1,2}\psi, N_{2,1}\psi\rangle_{L^2(\Gamma_T)} = s\lambda \int_{\Gamma_T} (\partial_\nu^A \eta^0 \xi \psi) \dv_\Gamma (D(x)\nabla_\Gamma \psi) \,\d S\,\d t\\
&=- s \lambda \int_{\Gamma_T} \langle D(x)\nabla_\Gamma \psi,\nabla_\Gamma(\partial_\nu^A \eta^0 \xi \psi)\rangle_\Gamma \,\d S\,\d t \\
&= - s \lambda \int_{\Gamma_T} \xi \psi \langle D(x)\nabla_\Gamma \psi, \nabla_\Gamma(\partial_\nu^A \eta^0)\rangle_\Gamma \d S \d t - s \lambda \int_{\Gamma_T} \partial_\nu^A \eta^0 \xi \langle D(x)\nabla_\Gamma \psi, \nabla_\Gamma \psi \rangle_\Gamma \d S \d t\\
& \geq - s \lambda \int_{\Gamma_T} \xi \psi \langle D(x)\nabla_\Gamma \psi, \nabla_\Gamma(\partial_\nu^A \eta^0)\rangle_\Gamma \,\d S\,\d t -\beta_0 s \lambda \int_{\Gamma_T} \partial_\nu^A \eta^0 \xi |\nabla_\Gamma \psi|_\Gamma^2 \d S\, \d t\\
& \geq - s \lambda \int_{\Gamma_T} \xi \psi \langle D(x)\nabla_\Gamma \psi, \nabla_\Gamma(\partial_\nu^A \eta^0)\rangle_\Gamma \,\d S\,\d t -C s \lambda \int_{\Gamma_T} \partial_\nu \eta^0 \xi |\nabla_\Gamma \psi|^2 \d S\, \d t,
\end{align*}
where we employed \eqref{ceqlm1}. The next terms are estimated by
\begin{align*}
\langle N_{1,1}\psi, N_{2,2}\psi\rangle_{L^2(\Gamma_T)} &=\frac{s}{2} \int_{\Gamma_T} \partial_t \alpha  \partial_t (\psi^2) \,\d S\,\d t \geq -C s \int_{\Gamma_T}  \xi^3 \psi^2 \,\d S\,\d t
\end{align*}
and by \eqref{2eq3.9} we have
\begin{align*}
\langle N_{1,2}\psi, N_{2,2}\psi\rangle_{L^2(\Gamma_T)} &=- s^2\lambda \int_{\Gamma_T} \partial_\nu^A \eta^0 (\partial_t \alpha) \xi \psi^2 \,\d S\,\d t \geq -C s^2 \lambda \int_{\Gamma_T} \xi^3 \psi^2 \,\d S\,\d t.
\end{align*}
Finally, the term
$$\langle N_{1,1}\psi, N_{2,3}\psi\rangle_{L^2(\Gamma_T)}= \int_{\Gamma_T} \partial_t \psi \partial_\nu^A \psi \,\d S\,\d t,$$
cancels with the one from \eqref{eqtg}, and
$$\langle N_{1,2}\psi, N_{2,3}\psi\rangle_{L^2(\Gamma_T)}=- s\lambda\int_{\Gamma_T} \xi \partial_\nu^A \eta^0 \partial_\nu^A \psi \psi \,\d S\,\d t.$$
\textbf{Step 3. The transformed estimate.} By regrouping final estimates in the previous steps and
increasing $\lambda_1$ and $s_1$ to absorb lower order terms, we derive
\begin{align*}
& \sum_{i,j=1}^N\langle  M_{1,i}\psi, M_{2,j}\psi\rangle_{L^2(\Omega_T)} + \sum_{i,j=1}^N\langle N_{1,i}\psi, N_{2,j}\psi\rangle_{L^2(\Gamma_T)}\\
& \quad \geq Cs^3\lambda^4\int_{\Omega_T} \xi^3\psi^2 \,\d x\,\d t- Cs^3\lambda^4\int_{(0,T)\times\omega'} \xi^3\psi^2 \,\d x\,\d t \\
& \qquad + C s^3\lambda^3\int_{\Gamma_T} \xi^3\psi^2 \,\d S\,\d t + Cs\lambda^2\int_{\Omega_T} \xi|\nabla\psi|^2 \,\d x\,\d t\\
& \qquad - Cs\lambda^2\int_{(0,T)\times\omega'}\xi|\nabla\psi|^2 \,\d x\,\d t -Cs\lambda^2\int_{\Gamma_T}\xi |\psi| (\partial_\nu \eta^0)^2 |\partial_\nu^A \psi| \,\d S\,\d t\\
& \qquad -s\lambda \int_{\Gamma_T} \xi(\partial_\nu \eta^0) (\partial_\nu^A \psi)^2 \,\d S\,\d t - Cs \lambda\int_{\Gamma_T} \xi\partial_\nu \eta^0 |\nabla_{\Gamma}\psi|^2  \,\d S\,\d t\\
& \qquad - Cs \lambda\int_{\Gamma_T}\xi  \psi  \langle \nabla_{\Gamma} (\partial_\nu^A \eta^0), D(x)\nabla_{\Gamma} \psi \rangle_{\Gamma} \,\d S\,\d t + Cs \lambda\int_{\Gamma_T} \xi\partial_\nu \eta^0 |\nabla_{\Gamma}\psi|^2  \,\d S\,\d t\\
& \qquad  -s \lambda\int_{\Gamma_T}  \xi \psi (\partial_\nu^A \eta^0) (\partial_\nu^A \psi) \,\d S\,\d t.
\end{align*}
We combine this estimate with \eqref{2eq3.8} and absorb lower order terms resulting from $\tilde{f}$ and $\tilde{g}$ to left-hand side by increasing $\lambda_1$ and $s_1$.
Using \eqref{eqlm1} and \eqref{ceqlm1},  we deduce
\begin{align}
&\lVert M_1\psi \rVert^2_{L^2(\Omega_T)}+  \lVert M_2\psi \rVert^2_{L^2(\Omega_T)} + \lVert N_1\psi \rVert^2_{L^2(\Gamma_T)}+ \lVert N_2\psi \rVert^2_{L^2(\Gamma_T)} \nonumber\\
& \quad + s^3\lambda^4\int_{\Omega_T} \xi^3\psi^2 \,\d x\,\d t + s\lambda^2\int_{\Omega_T} \xi|\nabla\psi|^2 \,\d x\,\d t + s^3\lambda^3\int_{\Gamma_T} \xi^3\psi^2 \,\d S\,\d t\nonumber\\
& \quad  + s\lambda\int_{\Gamma_T} \xi (|\nabla_{\Gamma}\psi|^2 + (\partial_\nu^A \psi)^2) \,\d S\,\d t\nonumber\\
& \leq C \int_{\Omega_T} e^{-2s\alpha}|\partial_t z -\dv(A(x)\nabla z)|^2 \,\d x\,\d t \nonumber\\
& \quad + C \int_{\Gamma_T} e^{-2s\alpha}|\partial_t z -\dv_{\Gamma}(D(x)\nabla_\Gamma z_\Gamma) +\partial_\nu^A z|^2 \,\d S\,\d t\nonumber\\
& \quad + Cs^3\lambda^4\int_{\omega'_T} \xi^3\psi^2 \,\d x\, \d t+  Cs\lambda^2\int_{\omega'_T}\xi|\nabla\psi|^2 \,\d x\,\d t \nonumber\\
& \quad + \underbrace{Cs\lambda^2\int_{\Gamma_T} (\partial_\nu \eta^0)^2 \xi |\psi| |\partial_\nu^A \psi| \,\d S\,\d t}_{I_1} + \underbrace{Cs \lambda\int_{\Gamma_T}  \xi |\partial_\nu \eta^0| |\nabla_\Gamma \psi|^2  \,\d S\,\d t}_{I_2} \nonumber\\
& \quad + \underbrace{C s \lambda\int_{\Gamma_T}\xi  \psi  |\nabla_{\Gamma}(\partial_\nu^A \eta^0)|_\Gamma |D(x)\nabla_{\Gamma} \psi|_\Gamma \,\d S\,\d t}_{I_3} + \underbrace{Cs \lambda\int_{\Gamma_T} \xi|\partial_\nu^A \eta^0| |\partial_\nu^A \psi| |\psi| \,\d S\,\d t}_{I_4}. \label{2eq3.13}
\end{align}
By Young's inequality, $I_1$ can be estimated by
\begin{align}
I_1 &\leq  C \int_{\Gamma_T} |s\lambda^{3/2}\xi^{1/2}\psi| \,|\lambda^{1/2}\xi^{1/2}\partial_\nu^A\psi| \,\d S\,\d t \nonumber\\
& \leq C \left(s^2\lambda^3 \int_{\Gamma_T} \xi^3\psi^2 \,\d S\,\d t + \lambda\int_{\Gamma_T}\xi (\partial_\nu^A \psi)^2 \,\d S\,\d t \right). \label{2eq3.15}
\end{align}
Choosing $s_1$ large enough, we can then control \eqref{2eq3.15} by the left-hand side of \eqref{2eq3.13}. In similar way, one can absorb $I_4$ and also $I_3$, since
\begin{align}
I_3 &\leq  C \left(\int_{\Gamma_T} \xi |\nabla_\Gamma\psi|^2 \,\d S\,\d t + s^2\lambda^2\int_{\Gamma_T}\xi \psi^2 \,\d S\,\d t \right).\label{2eq3.16}
\end{align}
Using the ellipticiy of $D$ with divergence formula \eqref{sdt} and the fact that $\xi(t,\cdot)$ is constant on $\Gamma$, the integral $I_2$ can be bounded by
\begin{align}
I_2 &\leq C s\lambda \int_{\Gamma_T} \xi |\nabla_\Gamma \psi|_\Gamma^2 \,\d S\,\d t \leq C s\lambda \int_0^T \xi \int_\Gamma\langle D(x)\nabla_\Gamma \psi, \nabla_\Gamma \psi\rangle_\Gamma \,\d S\,\d t\nonumber\\
& \leq C \int_0^T \int_\Gamma (s^{-1/2}\xi^{-1/2}|\dv_\Gamma( D(x)\nabla_\Gamma \psi)|)(s^{3/2}\lambda\xi^{3/2} |\psi|) \,\d S\,\d t \nonumber\\
& \leq s^{-1}\int_{\Gamma_T}\xi^{-1} |\dv_\Gamma( D(x)\nabla_\Gamma \psi)|^2 \,\d S\,\d t + Cs^3\lambda^2 \int_{\Gamma_T} \xi^3 \psi^2 \,\d S\,\d t. \label{2eq3.17}
\end{align}
The second addend in \eqref{2eq3.17} can be absorbed by the left-hand side of \eqref{2eq3.13}
by choosing $\lambda_1$ sufficiently large. Thus, we arrive at
\begin{align}
&\lVert M_1\psi \rVert^2_{L^2(\Omega_T)}+  \lVert M_2\psi \rVert^2_{L^2(\Omega_T)} + \lVert N_1\psi \rVert^2_{L^2(\Gamma_T)}+ \lVert N_2\psi \rVert^2_{L^2(\Gamma_T)} \nonumber\\
& \quad + s^3\lambda^4\int_{\Omega_T} \xi^3\psi^2 \,\d x\,\d t + s\lambda^2\int_{\Omega_T} \xi|\nabla\psi|^2 \,\d x\,\d t + s^3\lambda^3\int_{\Gamma_T} \xi^3\psi^2 \,\d S\,\d t\nonumber\\
& \quad  + s\lambda\int_{\Gamma_T} \xi (|\nabla_{\Gamma}\psi|^2 + (\partial_\nu^A \psi)^2) \,\d S\,\d t\nonumber\\
& \leq C \int_{\Omega_T} e^{-2s\alpha}|\partial_t z -\dv(A(x)\nabla z)|^2 \,\d x\,\d t \nonumber\\
& \quad + C \int_{\Gamma_T} e^{-2s\alpha}|\partial_t z -\dv_{\Gamma}(D(x)\nabla_\Gamma z_\Gamma) +\partial_\nu^A z|^2 \,\d S\,\d t + Cs^3\lambda^4\int_{\omega'_T} \xi^3\psi^2 \,\d x\,\d t \nonumber\\
& \quad +  Cs\lambda^2\int_{\omega'_T}\xi|\nabla\psi|^2 \,\d x\,\d t + s^{-1}\int_{\Gamma_T}\xi^{-1} |\dv_\Gamma( D(x)\nabla_\Gamma \psi)|^2 \,\d S\,\d t. \label{2eq3.18}
\end{align}
To transmit the last term in \eqref{2eq3.18} to the left, we observe first that\\
$-\dv_\Gamma( D(x)\nabla_\Gamma \psi)=N_2 \psi-s\psi \partial_t \alpha -\partial_\nu^A \psi$. Combined with \eqref{2eq3.9}, this identity yields
\begin{align}
I &:=s^{-1}\int_{\Gamma_T}\xi^{-1} |\dv_\Gamma( D(x)\nabla_\Gamma \psi)|^2 \,\d S\,\d t \nonumber\\
& \leq \frac{1}{2} \|N_2 \psi\|^2_{L^2(\Gamma_T)}+ C s \int_{\Gamma_T} \xi^3\psi^2 \,\d S\,\d t + C\int_{\Gamma_T} \xi (\partial_\nu^A \psi)^2 \,\d S\,\d t\label{2eq3.19}
\end{align}
for sufficiently large $s_1$. Choosing $\lambda_1$ and $s_1$ large enough so that
\eqref{2eq3.18} becomes
\begin{align}
&\lVert M_1\psi \rVert^2_{L^2(\Omega_T)}+  \lVert M_2\psi \rVert^2_{L^2(\Omega_T)} + \lVert N_1\psi \rVert^2_{L^2(\Gamma_T)}+ \lVert N_2\psi \rVert^2_{L^2(\Gamma_T)} \nonumber\\
& \quad + s^3\lambda^4\int_{\Omega_T} \xi^3\psi^2 \,\d x\,\d t + s\lambda^2\int_{\Omega_T} \xi|\nabla\psi|^2 \,\d x\,\d t + s^3\lambda^3\int_{\Gamma_T} \xi^3\psi^2 \,\d S\,\d t\nonumber\\
& \quad  + s\lambda\int_{\Gamma_T} \xi |\nabla_{\Gamma}\psi|^2 \,\d S\,\d t + s\lambda\int_{\Gamma_T} \xi (\partial_\nu^A \psi)^2 \,\d S\,\d t \nonumber\\
& \leq C \int_{\Omega_T} e^{-2s\alpha}|\partial_t z -\dv(A(x)\nabla z)|^2 \,\d x\,\d t \nonumber\\
& \qquad + C \int_{\Gamma_T} e^{-2s\alpha}|\partial_t z -\dv_{\Gamma}(D(x)\nabla_\Gamma z_\Gamma) +\partial_\nu^A z|^2 \,\d S\,\d t \nonumber\\
& \qquad + Cs^3\lambda^4\int_{\omega'_T} \xi^3\psi^2 \,\d x\,\d t + Cs\lambda^2\int_{\omega'_T}\xi|\nabla\psi|^2 \,\d x\,\d t. \label{2eq3.20}
\end{align}
The rest of the proof follows from the same strategy as  in \cite{MMS'17}.
\end{proof}

\begin{remark}
By means of the transformation $t'=T^{-1}(T-t_0) t+t_0$, the Carleman estimate \eqref{car1} remains true replacing $t(T-t)$ by $(t-t_0)(T -t)$ in the weight functions $\alpha$ and $\xi$ defined by \eqref{w1}, and integrating on $(t_0,T)$ instead of $(0,T)$, for $t_0 \in (0,T)$. In that case, we adopt the same notation for $\alpha$ and $\xi$, and we further denote $ \Omega_{t_0,T}:=(t_0,T)\times \Omega, \quad \Gamma_{t_0,T}:=(t_0,T)\times \Gamma, \quad\omega_{t_0,T}:=(t_0,T)\times \omega.$
\label{rmq2.13}
\end{remark}

\section{Global Lipschitz stability for an inverse source problem}\label{sec3}
The object of this section is to recover the source term $\mathcal{F}=(F,G)$ in \eqref{eq1to4} belonging to $\mathcal{S}(C_0)$ defined in \eqref{eqas}, from a single measurement $Y(T_0,\cdot)=(y,y_{\Gamma})\rvert_{t=T_0}$ and some extra partial observation on the first component of the solution  $y\rvert_{\omega_{t_0,T}}$. We notice here that the set of admissible source terms $\mathcal{S}(C_0)$ is necessarily involved, since the uniqueness for inverse source problems falls into default in the general case (see e.g., \cite[Commentary 6.6.]{IS'90}).
 
The main result of this paper reads as follows.
\begin{theorem}\label{thm1}
Let $T>0$, $t_0\in (0,T)$ and  $\displaystyle T_0=\frac{T+t_0}{2}$. Consider $Y :=(y,y_\Gamma)$ the mild solution of \eqref{eq1to4} and $C_0>0$. Then, there exists a positive constant \\$C=C(\Omega, \omega, T,t_0,C_0,\|B\|_\infty,\|p\|_\infty,\|b\|_\infty, \|q\|_\infty)$ such that, for any admissible source $\mathcal{F}=(F,G) \in \mathcal{S}(C_0)$, we have
\begin{equation}
\|(F,G)\|_{\mathbb{L}^{2}_T} \leq C\left(\|Y(T_0,\cdot)\|_{\mathbb{H}^2} + \|\partial_t y\|_{L^2(\omega_{t_0, T})}\right). \label{eq3.7}
\end{equation} \label{thm3.1}
\end{theorem}

\begin{proof}
Following Remark \ref{rmq2.13}, we may apply Carleman estimate \eqref{car1} on the interval $(t_0, T)$ instead of $(0,T)$. Throughout the proof, $C$ will denote a generic constant which is independent of $Y$. It may vary even from line to line.
The terms appearing in \eqref{eq3.7} are well defined, indeed, as mentioned in Section \ref{sec2}, we have then $Y:=(y,y_{\Gamma})\in \mathbb{E}_1(t_0, T)$.
The functions $z=\partial_t y$ and $z_\Gamma=\partial_t y_\Gamma$, where $(y,y_\Gamma)$ is the solution of \eqref{eq1to4}, are solutions of the system
\begin{small}
\begin{empheq}[left = \empheqlbrace]{alignat=2}
\begin{aligned}
&\partial_t z - \dv(A(x) \nabla z) + B(x)\cdot \nabla z + p(x)z = F_t(t,x) &\quad\text{in } \Omega_T , \\
&\partial_t z_{\Gamma} - \dv_{\Gamma} (D(x)\nabla_\Gamma z_{\Gamma}) +\partial_{\nu}^A z + \langle b(x), \nabla_\Gamma z_{\Gamma} \rangle_{\Gamma} + q(x)z_{\Gamma} = G_t(t,x) &\quad\text{on } \Gamma_T, \\
&z_{\Gamma}(t,x) = z\rvert_{\Gamma}(t,x) &\quad\text{on } \Gamma_T, \label{eq3.8to3.10}
\end{aligned}
\end{empheq}
\end{small}
and we have
\begin{small}
\begin{empheq}[left = \empheqlbrace]{alignat=2}
& z(T_0) - \dv(A\nabla y(T_0)) + B\cdot \nabla y(T_0) + p y(T_0) = F(T_0) \label{eq3.11}\\
& z_{\Gamma}(T_0) -\dv_{\Gamma} (D\nabla_\Gamma y_{\Gamma}(T_0)) + \partial_{\nu}^A y(T_0) + \langle b, \nabla_\Gamma y_{\Gamma}(T_0)\rangle_{\Gamma} + q y_{\Gamma}(T_0)= G(T_0).\label{eq3.12}
\end{empheq}
\end{small}
Since $(F,G)\in H^1(0,T; \mathbb{L}^2)$, by Proposition \ref{prop2} we have $(z, z_\Gamma) \in \mathbb{E}_1(t_0, T)$. Hence, we may apply Carleman estimate to \eqref{eq3.8to3.10} to obtain
\begin{small}
\begin{align}
& \int_{\Omega_{t_0,T}} \left(\frac{1}{s\xi} |\partial_t z|^2 + s^3\lambda^4 \xi^3 |z|^2 \right)e^{-2s\alpha} \d x \d t + \int_{\Gamma_{t_0,T}} \left(\frac{1}{s\xi} |\partial_t z_\Gamma|^2 + s^3\lambda^3\xi^3 |z_\Gamma|^2 \right)e^{-2s\alpha} \d S \d t \nonumber\\
& \quad\leq C s^3\lambda^4\int_{\omega_{t_0,T}} e^{-2s\alpha} \xi^3|z|^2 \d x \d t +C \int_{\Omega_{t_0,T}} e^{-2s\alpha}|F_t|^2 \d x \d t + C \int_{\Gamma_{t_0,T}} e^{-2s\alpha} |G_t|^2 \d S \d t, \label{eqq3.12}
\end{align}
\end{small}
for any $s>0$ large enough.
Since $\mathcal{F}=(F,G)\in \mathcal{S}(C_0)$,  we have
\begin{small}
\begin{align}
& \int_{\Omega_{t_0,T}} \left(\frac{1}{s\xi} |\partial_t z|^2 + s^3\lambda^4\xi^3 |z|^2 \right)e^{-2s\alpha} \d x \d t + \int_{\Gamma_{t_0,T}} \left(\frac{1}{s\xi} |\partial_t z_\Gamma|^2 + s^3\lambda^3\xi^3 |z_\Gamma|^2 \right)e^{-2s\alpha} \d S \d t \nonumber\\
& \quad\leq  C s^3\lambda^4\int_{\omega_{t_0,T}} e^{-2s\alpha} \xi^3|z|^2 \,\d x\,\d t +C \int_{\Omega_{t_0,T}} e^{-2s\alpha}\left|F\left(T_0,x\right)\right|^2 \,\d x\,\d t\label{eq3.13}\\
& \qquad + C \int_{\Gamma_{t_0,T}} e^{-2s\alpha} \left|G\left(T_0,x\right)\right|^2 \,\d S\,\d t. \nonumber
\end{align}
\end{small}
From \eqref{eq3.11}-\eqref{eq3.12}, to estimate the term
\begin{equation*}
\int_\Omega \left|F\left(T_0,x\right)\right|^2 e^{-2s\alpha\left(T_0,x\right)} \,\d x + \int_\Gamma \left|G\left(T_0,x\right)\right|^2 e^{-2s\alpha\left(T_0,x\right)} \,\d S,
\end{equation*}
we have to estimate the term
\begin{equation*}
\int_\Omega \left|z\left(T_0,x\right)\right|^2 e^{-2s\alpha\left(T_0,x\right)} \,\d x + \int_\Gamma \left|z_\Gamma\left(T_0,x\right)\right|^2 e^{-2s\alpha\left(T_0,x\right)} \,\d S.
\end{equation*}
Fix $x\in \Omega$ and take $\displaystyle H(t)=\int_\Omega \left|z\left(t,x\right)\right|^2 e^{-2s\alpha(t,x)} \,\d x$, for $t\in (0,T)$.\\
Since $\partial_t (z^2 e^{-2s\alpha})=(2\partial_t z\, z-2s\partial_t\alpha\, z^2)e^{-2s\alpha}\in L^2(\Omega_T)$ and $\lim\limits_{t\to t_0} e^{-2s\alpha(t,x)}=0$ for $x$ in $\overline{\Omega}$, we can differentiate $H$ under the integral sign. We further have
\begin{equation*}
\int_\Omega \left|z\left(T_0,x\right)\right|^2 e^{-2s\alpha\left(T_0,x\right)} \,\d x = \int_{t_0}^{T_0} \frac{\partial}{\partial t} \left(\int_\Omega \left|z\left(t,x\right)\right|^2 e^{-2s\alpha(t,x)} \,\d x\right)\,\d t
\end{equation*}
\begin{equation*}
\begin{aligned}
&=& \int_{t_0}^{T_0} \int_\Omega \left(2\partial_t y(t,x)\partial_t^2 y(t,x)- 2s(\partial_t \alpha)|\partial_t y(t,x)|^2 \right) e^{-2s\alpha} \,\d x\,\d t\\
&\leq& \int_{\Omega_{t_0,T}} \left(2|\partial_t y(t,x)|\,|\partial_t^2 y(t,x)| +Cs\xi^2 |\partial_t y(t,x)|^2 \right) e^{-2s\alpha} \,\d x\,\d t,
\end{aligned}
\end{equation*}
where we employed $|\partial_t \alpha| \leq C\xi^2$. On the other hand, we have
\begin{align*}
2|\partial_t y(t,x)|\,|\partial_t^2 y(t,x)| &= 2\frac{1}{s\sqrt{\xi}} |\partial_t^2 y(t,x)|s\sqrt{\xi}|\partial_t y(t,x)|\\
&\leq \frac{1}{s^2\xi} |\partial_t^2 y(t,x)|^2 + s^2\xi |\partial_t y(t,x)|^2\\
&\leq C\left(\frac{1}{s^2\xi} |\partial_t^2 y(t,x)|^2 + s^2 \lambda^4\xi^3 |\partial_t y(t,x)|^2\right),
\end{align*}
for large $\lambda$, using $\xi\leq C\xi^2$.
Hence, 
\begin{equation}
\int_\Omega \left|z\left(T_0,x\right)\right|^2 e^{-2s\alpha\left(T_0,x\right)} \,\d x \leq C\int_{\Omega_{t_0,T}} \left(\frac{1}{s^2\xi} |\partial_t^2 y|^2 + s^2\lambda^4\xi^3 |\partial_t y|^2 \right)e^{-2s\alpha} \,\d x\,\d t. \label{E1}
\end{equation}
Similarly, we have
\begin{align*}
2|\partial_t y_\Gamma(t,x)|\,|\partial_t^2 y_\Gamma(t,x)| &= 2\frac{1}{s\sqrt{\xi}} |\partial_t^2 y_\Gamma(t,x)|s\sqrt{\xi}|\partial_t y_\Gamma(t,x)|\\
&\leq \frac{1}{s^2\xi} |\partial_t^2 y_\Gamma(t,x)|^2 + s^2\xi |\partial_t y_\Gamma(t,x)|^2\\
&\leq C\left(\frac{1}{s^2\xi} |\partial_t^2 y_\Gamma(t,x)|^2 + s^2 \lambda^3\xi^3 |\partial_t y_\Gamma(t,x)|^2\right),
\end{align*}
and
\begin{equation}
\int_\Gamma \left|z_\Gamma\left(T_0,x\right)\right|^2 e^{-2s\alpha\left(T_0,x\right)} \,\d S \leq C\int_{\Gamma_{t_0,T}} \left(\frac{1}{s^2\xi} |\partial_t^2 y_\Gamma|^2 + s^2\lambda^3\xi^3 |\partial_t y_\Gamma|^2 \right)e^{-2s\alpha} \,\d S\,\d t, \label{E2}
\end{equation}
for large $\lambda$.
Adding inequalities \eqref{E1}, \;\eqref{E2} and applying \eqref{eq3.13}, we obtain
\begin{align}
& \int_\Omega \left|z\left(T_0,x\right)\right|^2 e^{-2s\alpha\left(T_0,x\right)} \,\d x + \int_\Gamma \left|z_\Gamma\left(T_0,x\right)\right|^2 e^{-2s\alpha\left(T_0,x\right)} \,\d S \nonumber\\
& \leq C\int_{\Omega_{t_0,T}} \left(\frac{1}{s^2\xi} |\partial_t^2 y|^2 + s^2\lambda^4\xi^3 |\partial_t y|^2 \right)e^{-2s\alpha} \,\d x\,\d t \nonumber\\
& \quad + C\int_{\Gamma_{t_0,T}} \left(\frac{1}{s^2\xi} |\partial_t^2 y_\Gamma|^2 + s^2\lambda^3\xi^3 |\partial_t y_\Gamma|^2 \right)e^{-2s\alpha} \,\d S\,\d t \nonumber\\
& \leq \frac{C}{s} \int_{\Omega_{t_0,T}} e^{-2s\alpha} \left|F\left(T_0,x\right)\right|^2 \,\d x\,\d t + \frac{C}{s} \int_{\Gamma_{t_0,T}} e^{-2s\alpha} \left|G\left(T_0,x\right)\right|^2 \,\d S\,\d t \nonumber\\
& \quad + \,Cs^2\lambda^4 \int_{\omega_{t_0,T}} e^{-2s\alpha}\xi^3 |\partial_t y|^2 \,\d x\,\d t. \label{EE2}
\end{align}
Since the coefficients of $A$ are bounded, $B\in L^\infty(\Omega)^N$ and $p\in L^\infty(\Omega)$, we obtain
\begin{equation}
\begin{aligned} 
&\int_\Omega \left(\left|\dv(A\nabla y\left(T_0,\cdot\right))\right|^2 + \left|B\cdot\nabla y\left(T_0,\cdot\right)\right|^2 + p^2 \left|y\left(T_0,\cdot\right)\right|^2\right) e^{-2s\alpha\left(T_0,\cdot\right)} \,\d x & \\
& \qquad \leq C \left\Vert y\left(T_0,\cdot\right)\right\Vert_{H^2(\Omega)}^2 . \label{EE1}
\end{aligned}
\end{equation}
Analogously, we have
\begin{align} 
& \int_\Gamma \left(\left|\dv_\Gamma (D \nabla_\Gamma y_\Gamma\left(T_0,\cdot\right))\right|^2 + \left|\langle b, \nabla_\Gamma y\left(T_0,\cdot\right) \rangle_\Gamma\right|^2 + q^2 \left|y_\Gamma\left(T_0,\cdot\right)\right|^2\right) e^{-2s\alpha\left(T_0,\cdot\right)} \,\d S \nonumber\\
& \qquad + \int_\Gamma \left|\partial_\nu^A y\left(T_0,x\right)\right|^2 e^{-2s\alpha\left(T_0,x\right)}\d S \nonumber\\
& \quad \leq C \left(\left\Vert y_\Gamma\left(T_0,\cdot\right)\right\Vert_{H^2(\Gamma)}^2 + \left\Vert \partial_\nu^A y\left(T_0,\cdot\right)\right\Vert_{L^2(\Gamma)}^2 \right) \nonumber\\
& \quad \leq C \left(\left\Vert y_\Gamma\left(T_0,\cdot\right)\right\Vert_{H^2(\Gamma)}^2 + \left\Vert y\left(T_0,\cdot\right)\right\Vert_{H^2(\Omega)}^2 \right), \label{EEE2}
\end{align}
using  $\|\partial_\nu^A y(T_0,\cdot)\|_{L^2(\Gamma)} \leq C\|y(T_0,\cdot)\|_{H^2(\Omega)}$, for some positive constant $C>0$, which holds by trace theorem since $A\in C(\overline{\Omega}; \mathbb{R}^{N\times N})$ (see Chapter 1, Theorem 8.3 in \cite{LM'72}).
Combining estimates \eqref{EE1} and \eqref{EEE2}, we obtain
\begin{small}
\begin{align}
& \int_\Omega \left(\left|\dv(A\nabla y\left(T_0,\cdot\right))\right|^2 + \left|B\cdot\nabla y\left(T_0,\cdot\right)\right|^2 + p^2 \left|y\left(T_0,\cdot\right)\right|^2\right) e^{-2s\alpha\left(T_0,\cdot\right)} \,\d x \\
& \quad + \int_\Gamma \left|\partial_\nu^A y\left(T_0,x\right)\right|^2 e^{-2s\alpha\left(T_0,x\right)} \,\d S \nonumber\\
& \quad + \int_\Gamma \left(\left|\dv_\Gamma (D \nabla_\Gamma y_\Gamma\left(T_0,\cdot\right))\right|^2 + \left|\langle b, \nabla_\Gamma y\left(T_0,\cdot\right) \rangle_\Gamma\right|^2 + q^2 \left|y_\Gamma\left(T_0,\cdot\right)\right|^2\right) e^{-2s\alpha\left(T_0,\cdot\right)} \,\d S \nonumber\\
& \quad\leq C \left(\left\Vert y\left(T_0,\cdot\right)\right\Vert_{H^2(\Omega)}^2 + \left\Vert y_\Gamma\left(T_0,\cdot\right)\right\Vert_{H^2(\Gamma)}^2\right)= C\left\Vert Y\left(T_0,\cdot\right)\right\Vert_{\mathbb{H}^2}^2 .\nonumber
\end{align}
\end{small}
Using \eqref{EE2} and \eqref{eq3.11}-\eqref{eq3.12},  we deduce
\begin{align}
& \int_\Omega e^{-2s\alpha\left(T_0,x\right)} \left|F\left(T_0,x\right)\right|^2 \,\d x + \int_\Gamma e^{-2s\alpha\left(T_0,x\right)} \left|G\left(T_0,x\right)\right|^2 \,\d S \nonumber\\
& \leq \frac{C}{s} \int_{\Omega_{t_0,T}} e^{-2s\alpha(t,x)} \left|F\left(T_0,x\right)\right|^2 \,\d x\,\d t + \frac{C}{s} \int_{\Gamma_{t_0,T}} e^{-2s\alpha(t,x)} \left|G\left(T_0,x\right)\right|^2 \,\d S\,\d t \nonumber\\
& \quad + \,C s^2\lambda^4 \int_{\omega_{t_0,T}} e^{-2s\alpha(t,x)} \xi^3 |\partial_t y|^2 \,\d x\,\d t +  C\left\Vert Y\left(T_0,\cdot\right)\right\Vert_{\mathbb{H}^2}^2 .\label{EE3}
\end{align}
Since $\alpha(t,x)\geq \alpha(T_0,x)$, for all $(t,x)\in \overline{\Omega}_{t_0, T}$, we have 
\begin{align}
&\int_{\Omega_{t_0,T}} e^{-2s\alpha(t,x)} \left|F\left(T_0,x\right)\right|^2 \,\d x\,\d t + \int_{\Gamma_{t_0,T}} e^{-2s\alpha(t,x)} \left|G\left(T_0,x\right)\right|^2 \,\d S\,\d t \nonumber\\
& \leq (T-t_0) \left(\int_\Omega e^{-2s\alpha\left(T_0,x\right)} \left|F\left(T_0,x\right)\right|^2 \,\d x + \int_\Gamma e^{-2s\alpha\left(T_0,x\right)} \left|G\left(T_0,x\right)\right|^2 \,\d S \right). \label{EE4}
\end{align}
From \eqref{EE3} and \eqref{EE4},  we obtain
\begin{align*}
&\int_{\Omega} e^{-2s\alpha\left(T_0,x\right)} \left|F\left(T_0,x\right)\right|^2 \,\d x + \int_{\Gamma} e^{-2s\alpha\left(T_0,x\right)} \left|G\left(T_0,x\right)\right|^2 \,\d S \nonumber\\
& \leq \frac{C}{s} \left(\int_{\Omega} e^{-2s\alpha\left(T_0,x\right)} \left|F\left(T_0,x\right)\right|^2 \,\d x + \int_{\Gamma} e^{-2s\alpha\left(T_0,x\right)} \left|G\left(T_0,x\right)\right|^2 \,\d S\right)\\
& \qquad + C\left\Vert Y\left(T_0,\cdot\right)\right\Vert_{\mathbb{H}^2}^2 + C s^2\lambda^4 \int_{\omega_{t_0,T}} e^{-2s\alpha(t,x)}\xi^3 |\partial_t y|^2 \,\d x\,\d t,
\end{align*}
and then
\begin{align}
&\left(1-\frac{C}{s}\right)\left(\int_{\Omega} e^{-2s\alpha\left(T_0,x\right)} \left|F\left(T_0,x\right)\right|^2 \,\d x + \int_{\Gamma} e^{-2s\alpha\left(T_0,x\right)} \left|G\left(T_0,x\right)\right|^2 \,\d S \right)\nonumber\\
& \qquad\leq C\left\Vert Y\left(T_0,\cdot\right)\right\Vert_{\mathbb{H}^2}^2 + C s^2\lambda^4 \int_{\omega_{t_0,T}} e^{-2s\alpha(t,x)}\xi^3 |\partial_t y|^2 \,\d x\,\d t. \label{EE5}
\end{align}
Since $\mathcal{F}=(F,G)\in \mathcal{S}(C_0)$, depending on $t\geq T_0$ or $\displaystyle t\leq T_0$, we have
\begin{align}
|F(t,x)|\leq \left|F\left(T_0,x\right)\right|+ \left|\int_{T_0}^t F_\tau(\tau,x) \,\d \tau \right| \leq C\left|F\left(T_0,x\right)\right|, \quad \forall (t,x)\in \Omega_T , \label{EE6}\\
|G(t,x)|\leq \left|G\left(T_0,x\right)\right|+ \left|\int_{T_0}^{t} G_\tau(\tau,x) \,\d \tau \right| \leq C\left|G\left(T_0,x\right)\right|, \quad \forall (t,x)\in \Gamma_T .\label{EE7}
\end{align}
Using \eqref{EE6}-\eqref{EE7} with \eqref{EE5}, we derive
\begin{align*}
&\left(1-\frac{C}{s}\right)\left(\int_{\Omega_{T}} e^{-2s\alpha\left(T_0,x\right)} |F(t,x)|^2 \,\d x\,\d t + \int_{\Gamma_{T}} e^{-2s\alpha\left(T_0,x\right)} |G(t,x)|^2 \,\d S\,\d t \right)\nonumber\\
& \qquad\leq C\left\Vert Y\left(T_0,\cdot\right)\right\Vert_{\mathbb{H}^2}^2 + C s^2\lambda^4 \int_{\omega_{t_0,T}} e^{-2s\alpha(t,x)}\xi^3 |\partial_t y|^2 \,\d x\,\d t.
\end{align*}
The functions $x \mapsto e^{-2s\alpha\left(T_0,x\right)}$ and $(t,x)\mapsto e^{-2s\alpha(t,x)}\xi^3(t,x)$ are bounded on $\Omega$ and $\Omega_{t_0,T}$, respectively. Fixing $\lambda,s>0$ sufficiently large, we obtain
\begin{align*}
\|(F,G)\|_{\mathbb{L}^2_{T}}^2 \leq C\left(\left\Vert Y\left(T_0,\cdot\right)\right\Vert_{\mathbb{H}^2}^2 + \|\partial_t y\|_{L^2(\omega_{t_0,T})}^2\right).
\end{align*}
Thus, the proof of Theorem \ref{thm1} is complete.
\end{proof}

\begin{remark}
We emphasize that in our inverse problem, the cases $T_0=0$ or $T_0 =T$, where $T_0$ is the time of observation, are not considered. In fact, the weight functions used in Carleman estimate blow up as $t \to 0$ and as $t \to T$, and in the proof we used the boundedness of these functions away from $0$ and $T$. Then, it would be of much interest to prove similar results in these cases.
\end{remark}

\subsection{Uniqueness and stability in a particular case}
A particular but interesting case of inverse source problems is when the source terms in \eqref{eq1to4} are given by
\begin{align}
F(t,x) &=f(x)r(t,x), \quad \text{ for all } (t,x)\in \Omega_T, \label{Eqq5}\\
G(t,x) &=g(x)\widetilde{r}(t,x), \quad \text{ for all } (t,x)\in \Gamma_T. \label{Eqq6}
\end{align}
Here, the inverse source problem is to determine the couple of $x$-dependent sources $(f,g)$, by means of a single measurement $Y(T_0, \cdot)=(y,y_{\Gamma})\rvert_{t=T_0}$ and a partial observation $y\rvert_{\omega_{t_0,T}}$, provided that the couple of $(t,x)$-dependent functions $(r,\widetilde{r})$ belonging to $\mathcal{C}^{1,0}$ are known and satisfying
\begin{align}
r(T_0,x) &\neq 0, \quad x\in \overline{\Omega},\label{Eq5}\\
\widetilde{r}(T_0,x)&\neq 0, \quad x\in \Gamma.\label{Eq6}
\end{align}

\begin{remark}
Under assumptions \eqref{Eq5}-\eqref{Eq6}, one can check that the source term $(F,G)$ in \eqref{Eqq5}-\eqref{Eqq6} belongs to the set of admissible sources $\mathcal{S}(C_0)$, for some positive constant $C_0=C(r,\widetilde{r})$. In fact, by \eqref{Eq5} we have
\begin{equation}
|r(T_0,x)| \geq r_0 >0, \text{ for all } x\in \overline{\Omega}, \label{Imin}
\end{equation}
where $r_0 =\min\limits_{\overline{\Omega}} |r(T_0,\cdot)|$. Hence, 
\begin{align*}
|F_t(t,x)| = |f(x)r_t(t,x)| &\leq (\sup_{\overline{\Omega}_T}|r_t|)|f(x)|\\
& \leq \frac{1}{r_0} (\sup_{\overline{\Omega}_T} |r_t|) |f(x)||r(T_0,x)|=C_1 |F(T_0,x)|.
\end{align*}
In a similar way,  we obtain $|G_t(t,x)| \leq C_2 |G(T_0,x)|$.\label{rmk3.4}
\end{remark}
Thus, as a consequence of Theorem \ref{thm1} we have the following.
\begin{proposition}\label{pR1}
Let $Y_0 \in \mathbb{L}^2$. Assume that $(r,\widetilde{r})\in \mathcal{C}^{1,0}$ satisfies \eqref{Eq5}-\eqref{Eq6}. Then, there exists a constant $C=C(\Omega, \omega, T,t_0,C_0,\|B\|_\infty,\|p\|_\infty,\|b\|_\infty, \|q\|_\infty, r, \widetilde{r}) > 0$ such that for all $(f_1,g_1),(f_2,g_2)\in \mathbb{L}^2$,
\begin{equation}
\|(f_1 -f_2, g_1 -g_2)\|_{\mathbb{L}^2} \leq C\left(\|Y_{1}(T_0,\cdot) -Y_{2}(T_0,\cdot)\|_{\mathbb{H}^2} + \|\partial_t y_1 - \partial_t y_2\|_{L^2(\omega_{t_0,T})}\right), \label{Isu}
\end{equation}
where $Y_1$ and $Y_2$ are the mild solutions of \eqref{eq1to4} respectively associated to $(f_1 r,g_1 \,\widetilde{r})$ and $(f_2 r,g_2 \,\widetilde{r})$.
\end{proposition}

\begin{remark}
If the known parts $r$ and $\widetilde{r}$ of the source terms in \eqref{Eqq5}-\eqref{Eqq6} are assumed to be positive as in \cite[Theorem III.2.]{Va'11}, Proposition \ref{pR1} follows directly from Theorem \ref{thm1}. In our case, we assume that $r$ and $\widetilde{r}$ are only positive on $T_0$ which makes the situation a bit more complicated.
\end{remark}

\begin{proof}
Let $f=f_1 -f_2$, $g=g_1-g_2$ and $Y=Y_1-Y_2$ the corresponding solution. By Remark \ref{rmk3.4} we have $\mathcal{F}=(f r,g\widetilde{r}) \in \mathcal{S}(C_0)$, for some $C_0=C(r,\widetilde{r})>0$. Hence, Theorem \ref{thm1} implies that there exists $C=C(\Omega, \omega, T,t_0,\|a\|_\infty,\|b\|_\infty, r, \widetilde{r}) > 0$ such that
\begin{equation}
\|(f r, g \,\widetilde{r})\|_{\mathbb{L}^2_T} \leq C\left(\|Y(T_0,\cdot)\|_{\mathbb{H}^2} + \|\partial_t y_1 \|_{\omega_{t_0,T}}\right). \label{Ifr}
\end{equation}
To obtain \eqref{Isu}, it suffices to prove that there exist $c_0>0$ and $\tau>0$ such that $(T_0-\tau,T_0+\tau)\subset (0,T)$ and,
$$\forall (t,x)\in (T_0-\tau,T_0+\tau)\times \overline{\Omega}\colon \; |r(t,x)|\geq c_0 >0,$$
$$\forall (t,x)\in (T_0-\tau,T_0+\tau)\times \Gamma \colon \; |\widetilde{r}(t,x)|\geq c_0 >0.$$
By \eqref{Imin} we have $|r(T_0,\cdot)|\geq r_0$ on $\overline{\Omega}$. Using the uniform continuity of $|r|$ on $[0,T]\times \overline{\Omega}$, there exists $\tau_1>0$ such that
$$|t_1-t_2|+\|x_1-x_2\|_1<\tau_1 \text{ implies }\displaystyle |r(t_2,x_2)|-\frac{r_0}{2} <|r(t_1,x_1)|.$$ 
We can choose $\tau$ such that $(T_0-\tau_1,T_0+\tau_1)\subset (0,T)$. It follows that,  for all $(t,x)\in (T_0-\tau_1,T_0+\tau_1)\times \overline{\Omega}$, we have
$$|r(t,x)|>|r(T_0,x)|-\frac{r_0}{2}\geq \frac{r_0}{2}.$$
Similarly, there exists $\tau_2>0$ such that $(T_0-\tau_2,T_0+\tau_2)\subset (0,T)$ and,  for all $(t,x)\in (T_0-\tau_2,T_0+\tau_2)\times \Gamma$,  we have $\displaystyle |\widetilde{r}(t,x)|\geq \frac{r_0}{2}$.
Let $\tau=\min(\tau_1,\tau_2)$ and $\displaystyle c_0=\frac{r_0}{2}$. Then,
$$\forall (t,x)\in (T_0-\tau,T_0+\tau)\times \overline{\Omega}\colon \; |r(t,x)|\geq c_0 >0,$$
$$\forall (t,x)\in (T_0-\tau,T_0+\tau)\times \Gamma \colon \; |\widetilde{r}(t,x)|\geq c_0 >0.$$
Furthermore,
$$\forall (t,x)\in (T_0-\tau,T_0+\tau)\times \overline{\Omega}\colon \; |f(x)|\leq \frac{1}{c_0} |f(x)| |r(t,x)|,$$
$$\forall (t,x)\in (T_0-\tau,T_0+\tau)\times \Gamma \colon \; |g(x)|\leq \frac{1}{c_0} |g(x)| |\widetilde{r}(t,x)|.$$
By integrating the previous inequalities we obtain,
\begin{align*}
\|(f,g)\|_{\mathbb{L}^2}^2 &=\int_\Omega |f(x)|^2 \,\d x + \int_\Gamma |g(x)|^2 \,\d S \\
& \leq \frac{1}{2 \tau c_0^2}\left( \int_{T_0-\tau}^{T_0+\tau} \int_\Omega |f(x)r(t,x)|^2 \,\d x\,\d t+ \int_{T_0-\tau}^{T_0+\tau} \int_\Gamma |g(x)\widetilde{r}(t,x)|^2 \,\d S\,\d t\right)\\
& \leq \frac{1}{2 \tau c_0^2}\left(\int_{\Omega_T} |f(x)r(t,x)|^2 \,\d x\,\d t+ \int_{\Gamma_T} |g(x)\widetilde{r}(t,x)|^2 \,\d S\,\d t\right)\\
&=\frac{1}{2 \tau c_0^2}  \|(f r, g \,\widetilde{r})\|_{\mathbb{L}^2_T}^2.
\end{align*}
Inequality \eqref{Ifr} allows to conclude.
\end{proof}

As an application of Proposition \ref{pR1}, we derive the following uniqueness result.
\begin{corollary}
Assume that $Y_1:=(y_1,y_{1,\Gamma})$ and $Y_2 :=(y_2,y_{2,\Gamma})$ are the mild solutions
of \eqref{eq1to4} respectively associated to $\mathcal{F}_1:=(f_1 r, g_1 \widetilde{r})$ and $\mathcal{F}_2:=(f_2 r, g_2 \widetilde{r})$, where $(r, \widetilde{r})\in \mathcal{C}^{1,0}$ satisfying \eqref{Eq5}-\eqref{Eq6}. If $Y_{1}(T_0,\cdot) =Y_{2}(T_0,\cdot)$ and $\partial_t y_1 = \partial_t y_2$ in $\omega_{t_0,T}$, then $f_1\equiv f_2$ in $\Omega$ and $g_1\equiv g_2$ in $\Gamma$.
\end{corollary}

\begin{remark}
Here, we dealt with the case when the time of observation $T_0$ is in the interval $(t_0,T)$, since better regularity results hold than $(0,T)$. However, one can obtain the same result when the observation is taken in $(0, t_0)$ but this requires more regularity on the initial conditions as well as source terms.
\end{remark}



\end{document}